\documentclass[a4paper,11pt]{amsart}
\usepackage{amsfonts,amsthm,amsmath,amssymb,graphicx}
\theoremstyle{plain}
\newtheorem{lem}{Lemma}[section]
\newtheorem{prop}[lem]{Proposition}
\newtheorem{thm}[lem]{Theorem}
\newtheorem{cor}[lem]{Corollary}
\theoremstyle{definition}

\theoremstyle{remark}

\DeclareMathOperator{\rank}{rank}

\DeclareMathOperator{\sym}{sym}
\DeclareMathOperator{\spn}{span}

\DeclareMathOperator{\e}{e}

\DeclareMathOperator{\diag}{diag}

\DeclareMathOperator{\modulo}{mod\ }

\newcommand{\bmu}{\boldsymbol \mu}
\newcommand{\bdelta}{\boldsymbol \delta}
\newcommand{\bbeta}{\boldsymbol \beta}

\newcommand{\Z}{\mathbb Z}
\newcommand{\Q}{\mathbb Q}
\newcommand{\A}{\mathbb A}
\newcommand{\F}{\mathbb F}
\newcommand{\E}{\mathbb E}

\newcommand{\R}{\mathbb R}
\newcommand{\C}{\mathbb C}
\newcommand{\stufe}{\mathcal N }
\newcommand{\G}{\mathcal G}

\newcommand{\h}{\mathfrak h}
\newcommand{\K}{\mathcal K}
\newcommand{\M}{\mathcal M}
\newcommand{\Y}{\mathcal Y}

\begin{document}

% Enter full title and short title for running headers
\title[Half-integral weight Siegel Eisenstein series]
{Hecke operators on half-integral weight Siegel Eisenstein series}
%\shorttitle{Degree 2 Eisenstein series}

% Author name(s)
\author{Lynne H. Walling}
\address{School of Mathematics, University of Bristol, University Walk, Clifton, Bristol BS8 1TW, United Kingdom;
phone +44 (0)117 331-5245, fax +44 (0)117 928-7978}
\email{l.walling@bristol.ac.uk}
% Abbreviated author name for running headers
%\abbrevauthor{L.H. Walling}
% Abbreviated author name for first page header
%\headabbrevauthor{Walling, L.H.}

%\address{School of Mathematics, University of Bristol, University Walk, Clifton, Bristol BS8 1TW, United Kingdom}

% Address / e-mail address of corresponding author
%\correspdetails{l.walling@bristol.ac.uk}

\keywords{Hecke eigenvalues, Eisenstein series, Siegel modular forms, half-integral weight}

\begin{abstract} 
We construct a basis for the space of half-integral weight Siegel Eisenstein series of level $4\stufe$ where $\stufe$ is odd and square-free.  Then we restrict our attention to those Eisenstein series generated from elements of $\Gamma_0(4)$, commenting on why this restriction is necessary for our methods.
We directly apply to these forms all Hecke operators attached to odd primes, and we realize the images explicitly as linear combinations of Siegel Eisenstein series.  Using this information, we diagonalize the subspace of Eisenstein series generated from elements of $\Gamma_0(4)$, obtaining a multiplicity-one result.
\end{abstract}

\maketitle
\def\thefootnote{}
\footnote{2010 {\it Mathematics Subject Classification}: Primary
11F46, 11F11 }
\def\thefootnote{\arabic{footnote}}

\section{Introduction} 

In a seminal paper \cite{Shim}, Shimura established a beautiful correspondence between (Siegel degree 1) cusp forms of half-integral weight $k/2$, level $4\stufe$ and character $\chi$, and elliptic modular forms of integral weight $k-1$, level $2\stufe$ and character $\chi^2$.  
Essentially the correspondence is established by comparing Hecke-eigenvalues (and using Weil's Converse Theorem to show that the integral weight forms constructed are indeed a modular forms).

In recent work \cite{int wt} we constructed a basis for the space of degree $n$, integral weight $k$, arbitrary level $\stufe$ and character $\chi$ Siegel Eisenstein series, and through direct computations we produced a basis of simultaneous eigenforms for the Hecke operators
$$\{T(p), T_j(p^2):\ 1\le j\le n,\ p \text{ prime},\ p\nmid\stufe\ \};$$
when $\stufe$ is square-free, the elements of this basis are also eigenforms for
$$\{T(q), T_j(q^2):\ 1\le j\le n,\ q \text{ prime},\ q|\stufe\ \}$$
and these basis elements are distinguished by their eigenvalues.

Here we extend this work to consider half-integral weight Siegel Eisenstein series.
There are several difficulties that arise, since we need to work with automorphy factors.
In principle we could work in a cover of the symplectic group, but following Shimura, for any matrix in the congruence subgroup $\Gamma_0(4)$ (defined below) we make a specific choice for an automorphy factor given by a quotient of Siegel theta series (also defined below).
Our computations take advantage of nice properties of these theta series and of generalized Gauss sums (Proposition 2.2).
Unfortunately, this also limits our detailed evaluation of Hecke operators to those Siegel Eisenstein series generated from elements in $\Gamma_0(4)$, and we are unable to give a satisfactory evaluation of the action of $T_j(4)$.

Like Shimura, we only consider levels $4\stufe$, and here we especially focus on the case of $\stufe$ odd and square-free.
In degree 1 this is fully justified as there is no one-fold covering group of a subgroup of $SL_2(\Z)$ that properly contains
$$\{[\gamma,\theta(\gamma\tau)/\theta(\tau)]:\ \gamma\in\Gamma_0(4)\ \}$$
(see, for example, Corollary 3.7).  
However, we are not able to prove this for degree $n>1$, although (also in Corollary 3.7) we prove a partial result toward this.
The reasoning used to prove Corollary 3.7 is also used to show that, regardless of the choice of automorphy factor for certain $\gamma\in Sp_n(\Z)$ with $\gamma\not\in\Gamma_0(4)$, the Siegel Eisenstein series generated from $\gamma$ is 0 (Proposition 3.6).
(In Proposition 3.5 we give necessary conditions on $\chi$ to have a nonzero Eisenstein series.)

For $\stufe$ odd and square-free, we show that the subspace of Eisenstein series generated from elements of $\Gamma_0(4)$ has a basis of simultaneous eigenforms for the Hecke operators
$$\{T_j(q^2):\ q \text{ prime},\ q|\stufe\ \},$$
and these basis elements are distinguished by their eigenvalues.
(As in the case degree 1,  the half-integral weight Hecke operator $T(p)$ is 0 for any prime $p$; see, for example, Proposition 2.1 \cite{half-int}).

When the degree $n$ is 1, we recover Shimura's correspondence \cite{Shim}:
With $\sigma=(\stufe_0,\stufe_1)$ varying over all multiplicative partitions of $\stufe$ (meaning that $\stufe_0\stufe_1=\stufe$),
we have a basis $\{\widetilde\E_{\sigma}\}$ for the space of weight $k/2$, level $4\stufe$ and character $\chi$ Eisenstein series generated from elements of $\Gamma_0(4)$, and a basis $\{\E'_{\sigma}\}$ for the space of weight $k-1$, level $\stufe$ and character $\chi^2$
Eisenstein series, so that for every odd prime $p$, the $T_1(p^2)$-eigenvalue of $\widetilde\E_{\sigma}$ is the $T(p)$-eigenvalue of $\E'_{\sigma}$.
For $n>1$, such a correspondence is unclear; below we exhibit the eigenvalues for half-integral weight and integral weight Eisenstein series.

We still assume that $\stufe$ is odd and square-free.
Take $k',k\in\Z_+$ with $k$ odd, $\chi'$ a character modulo $\stufe$, and $\chi$ a character modulo $4\stufe$.  
For $\sigma=(\stufe_0,\ldots,\stufe_n)$ a multiplicative partition of $\stufe$, we have corresponding Eisenstein series $\E'_{\sigma}$ and $\widetilde\E_{\sigma}$ of weights $k'$ and $k/2$, levels $\stufe$ and $4\stufe$, characters $\chi'$ and $\chi$ (respectively).
(Note that by Proposition 3.6 \cite{int wt} and Proposition 4.1, when $\E'_{\sigma}\not=0$ we have
$\left(\chi'_{\stufe/\stufe_0\stufe_n}\right)^2=1$ and when $\widetilde\E_{\sigma}\not=0$ we have 
$\left(\chi_{4\stufe/\stufe_0\stufe_n}\right)^2=1$.)
For a prime $q|\stufe$ and $0\le d\le n$ so that $q|\stufe_d$, by Corollary 4.3  \cite{int wt} we have
$$\E'_{\sigma}|T(q)=q^{k'd-d(d+1)/2}\chi'_{\stufe/q}(\overline qX_dM_{\sigma},X_d)\E'_{\sigma}$$
where, for each prime $q'|\stufe/q$ and $0\le d'\le n$ so that $q'|\stufe_{d'},$
$$\chi'_{q'}(\overline qX_dM_{\sigma},X_d)
=\begin{cases}\chi'_{q'}(q^{d-d'})&\text{if $d'\le d$,}\\
\chi'_{q'}(q^{d'-d})&\text{if $d'>d$.}\end{cases}$$
Also, by Corollary 4.5 \cite{int wt}, we have
\begin{align*}
\E'_{\sigma}|T_j(q^2)&=
q^{jd}\sum_{s=0}^j q^{s(2k'-2d+s-j-1)}\chi'_{\stufe'_0}(q^{2s})\chi'_{\stufe'_n}(q^{2(j-s)})\\
&\quad\cdot \bbeta_q(d,s)\bbeta_q(n-d,j-s)\E'_{\sigma}
\end{align*}
where $\bbeta_q(m,r)$ is the number of $r$-dimensional subspaces of an $m$-dimensional space over $\Z/q\Z$ and
$\stufe'_i=\stufe_i/(q,\stufe_i)$.
In contrast, by Corollary 4.4 we have
\begin{align*}
\widetilde\E_{\sigma}|T_j(q^2)
&=
q^{jd}\sum_{s=0}^j q^{s(k-2d+s-j-1)}
\chi_{\stufe_0'}(q^{2s})\chi_{\stufe_n'}(q^{2(j-s)})\\
&\quad\cdot \bbeta_q(d,s)\bbeta_q(n-d,j-s) \widetilde\E_{\sigma}.
\end{align*}
In particular,
$$\widetilde\E_{\sigma}|T_n(q^2)=q^{d(k-d-1)}\chi_{\stufe'_0}(q^{2d})\chi_{\stufe'_n}(q^{2(n-d)}).$$
For an odd prime $p\nmid \stufe$, by Corollaries 5.3 and 5.5 \cite{int wt} we have
$$\E'_{\sigma}|T(p)=\left(\prod_{0<d\le n}\chi'_{\stufe_d}(p^d)\right)
\prod_{i=1}^n\left(\chi'(p)\overline\chi'_{\stufe_n}(p^2)p^{k'-i}+1\right)\E'_{\sigma},$$
\begin{align*}
\E'_{\sigma}|T_j(p^2)
&=\bbeta_p(n,j)\sum_{r+s\le j}p^{k'(j-r+s)-(j-r)(n+1)}\chi'(p^{j-r+s})\chi'_{\stufe_n}(p^{2(r-s)})\\
&\quad\cdot \bbeta_p(j,r)\bbeta_p(j-r,s)\sym_p(j-r-s)\E'_{\sigma}
\end{align*}
where $\sym_p(\ell)$ is the number of symmetric, $\ell\times\ell$, invertible matrices over $\F=\Z/p\Z$.
In contrast, by Theorem 4.5 we have
\begin{align*}
\widetilde\E_{\sigma}|T_j(p^2)
&=\bbeta_p(n,j)\sum_{r+s\le j}p^{k(j-r+s)/2-(j-r)(n+1)}\chi(p^{j-r+s})\chi_{\stufe_n}(p^{2(r-s)})\\
&\quad\cdot \bbeta_p(j,r)\bbeta_p(j-r,s) \left(\frac{\G_1(p)}{\sqrt{p}}\right)^{j-r-s}\sym_p^{\psi}(j-r-s)\widetilde\E_{\sigma}
\end{align*}
where $\G_1(p)$ is the classical Gauss sum modulo $p$, $\psi(*)=\left(\frac{*}{p}\right)$, and
$\sym_p^{\psi}(\ell)=\sum_{U\in\F_{\sym}^{\ell,\ell}}\psi(\det U).$

Since these eigenvalue of $\widetilde\E_{\sigma}$ under $T_j(p^2)$ are not so attractive, in Corollary 4.6 we introduce an alternate set of generators for the local Hecke algebra, obtaining more attractive eigenvalues (similar to what we did in Corollary 5.5 \cite{int wt}).

As much as possible, we borrow results from \cite{int wt}.  Here the construction of the Siegel Eisenstein series is a bit different because of the automorphy factors involved.  As the automorphy factors contain Gauss sums, for the evaluation of the action of the Hecke operators we establish some nice identities between generalized Gauss sums (Propositions 5.1, 5.2, 5.3).  Although these identities can surely be established by other methods, we rely on changes of variables to provide elementary arguments.

The author thanks Andrew Booker and Fredrik Stromberg for helpful conversations.

\section{Preliminaries}

For $n\in\Z$, $n>1$, Siegel's degree $n$ upper half-space is defined as
$$\h_{(n)}=\{X+iY:\ X,Y\in\R^{n,n}_{\sym},\ Y>0\ \}$$
where $Y>0$ means that, as a quadratic form, $Y$ is positive definite.
The symplectic group $Sp_n(\Z)$ acts on $\h_{(n)}$, where
\begin{align*}
&Sp_n(\Z)\\
&=\left\{\begin{pmatrix}A&B\\C&D\end{pmatrix}\in SL_{2n}(\Z):\ A\,^tB=B\,^tA,\ C\,^tD=D\,^tC,\ A\,^tD-B\,^tD=I\ \right\}
\end{align*}
(here $^tB$ is the transpose of $B$).
For $\gamma=\begin{pmatrix}A&B\\C&D\end{pmatrix}\in Sp_n(\Z)$ and $\tau\in\h_{(n)}$, the action of $\gamma$ on $\tau$ is given by
$$\gamma\tau=(A\tau+B)(C\tau+D)^{-1}.$$
Note that for $\begin{pmatrix}A&B\\C&D\end{pmatrix}\in Sp_n(\Z)$, $(C\ D)$ is a coprime symmetric pair, meaning that $C\,^tD$ is symmetric and for all primes $p$, $\rank_p(C\ D)=n$ (here $\rank_p(C\ D)$ denotes the rank of the matrix $(C\ D)$ modulo $p$, meaning we view $(C\ D)$ as a matrix over $\Z/p\Z$).  Conversely, given $C,D\in\Z^{n,n}$ so that $(C\ D)$ is a coprime symmetric pair, there is a matrix $\begin{pmatrix}A&B\\C&D\end{pmatrix}$ in $Sp_n(\Z)$.  When $(C\ D)$ is a pair of integral $n\times n$ matrices, we write $(C,D)=1$ to mean that $C$ and $D$ are coprime.

To construct Siegel Eisenstein series of half-integral weight $k/2$, we need to make sense of $(\det(C\tau+D))^{-k/2}.$  Thus we have the following definitions.

\smallskip
\noindent{\bf Definition.}
An automorphy factor for $\gamma=\begin{pmatrix}A&B\\C&D\end{pmatrix}\in Sp_n(\Z)$ is an analytic function $\varphi_{_\gamma}(\tau)$ on $\h_{(n)}$ so that $|\varphi_{_\gamma}(\tau)|^2=|\det(C\tau+D)|$.
(Note that it is known that $\det(C\tau+D)\not=0$; see Proposition 1.2.1 \cite{And}.)
When we also have $\gamma'\in Sp_n(\Z)$ and $\varphi_{_{\gamma'}}(\tau)$ an automorphy factor for $\gamma'$, we have
$$[\gamma,\varphi_{_\gamma}(\tau)][\gamma,\varphi_{_{\gamma'}}(\tau)]=[\gamma\gamma',\varphi_{_\gamma}(\gamma'\tau)\varphi_{_{\gamma'}}(\tau)].$$
When $\det D\not=0$, we define $S_{C,D}(\tau)$ by taking
$$\lim_{\lambda\to0^+}S_{C,D}(i\lambda I)=\sqrt{\det D}\in\R_+\cup i\R_+$$
and extending analytically to $\tau\in \h_{(n)}$.  Thus with 
$\gamma=\begin{pmatrix}*&*\\C&D\end{pmatrix}\in Sp_n(\Z)$, $\det D\not=0$, and
$\varphi_{_\gamma}(\tau)$ an automorphy factor for $\gamma$, we have that $$\frac{\varphi_{_\gamma}(\tau)}{S_{C,D}(\tau)}$$
is analytic with absolute value 1, so $\varphi_{_\gamma}(\tau)=v(\gamma)S_{C,D}(\tau)$ for some $v(\gamma)$ with $|v(\gamma)|=1$.
We  define the basic degree $n$ Siegel theta series by
$$\theta(\tau)=\sum_{U\in\Z^{1,n}}\e\{2\,^tUU\tau\} \text{ where } \e\{*\}=\exp(\pi iTr(*)),$$
and for $(C\ D)$ a coprime symmetric pair  with $\det D\not=0$, we define a generalized Gauss sum by
$$\G_C(D)=\sum_{U\in\Z^{1,n}/\Z^{1,n}D}\e\{2\,^tUUD^{-1}C\}.$$

The following result is a special case of Theorem 1.3.13 and Proposition 1.4.5 \cite{And}.

\begin{prop} (Transformation Formula)
Set
$$\Gamma_0(4)=\left\{\begin{pmatrix}A&B\\C&D\end{pmatrix}\in Sp_n(\Z):\ 4|C\ \right\}.$$
For
$\gamma=\begin{pmatrix}A&B\\C&D\end{pmatrix}\in\Gamma_0(4)$, we have
$$\frac{\theta(\gamma\tau)}{\theta(\tau)}=
\frac{\overline\G_C(D)}{\sqrt{\det D}}\,S_{C,D}(\tau).$$
\end{prop}

Because of this result, we make the following definition.

\smallskip\noindent{\bf Definition.}
For $\gamma\in \Gamma_0(4)$, we set
$$\widetilde\gamma=[\gamma,\theta(\gamma\tau)/\theta(\tau)].$$
Note that with $\gamma,\delta\in\Gamma_0(4)$, we have $\widetilde\gamma\widetilde\delta=\widetilde{(\gamma\delta)}.$
\smallskip

The following identities will be useful.

\begin{prop}  Take
$\delta=\begin{pmatrix}A&B\\C&D\end{pmatrix}\in\Gamma_0(4)$.
\begin{enumerate}
\item[(a)]  For $\alpha\in\Gamma_{\infty}$, we have $\theta(\alpha\tau)=\theta(\tau)$, and for
$Y\in\Z^{n,n}_{\sym}$, 
we have 
$$\frac{\overline\G_C(D+CY)}{\sqrt{\det D+CY}}\,S_{C,D+CY}(\tau)
=\frac{\overline\G_C(D)}{\sqrt{\det D}}\,S_{C,D}(\tau+Y).$$
\item[(b)]  For $E\in SL_n(\Z)$, we have $S_{C,D}(\tau)=S_{EC,ED}(\tau).$
\item[(c)]  For $E\in GL_n(\Z)$, we have
$$\G_{EC}(ED)=\G_C(D)=\G_{CE}(D\,^tE^{-1}).$$
\end{enumerate}
\end{prop}

\begin{proof}
(a)  
For $\alpha\in\Gamma_{\infty}$, we have $\alpha=\begin{pmatrix}G&GY\\0&^tG^{-1}\end{pmatrix}$ for some $G\in GL_n(\Z)$ and $Y\in\Z^{n,n}_{\sym}$.  Thus
for $U\in\Z^{1,n}$, $^tUUGY\,^tG$ is integral, and $UG$ varies over $\Z^{1,n}$ as $U$ does. Hence
\begin{align*}
\theta(\alpha\tau)&=\sum_{U\in\Z^{1,n}}\e\{2\,^tUUG(\tau+Y)\,^tG\}\\
&=\sum_{U\in\Z^{1,n}}\e\{2\,^t(UG)(UG)\tau\}\\
&=\theta(\tau).
\end{align*}
So with $G=I$ and noting that $\delta\alpha\in\Gamma_0(4)$, by Proposition 2.1 we have
\begin{align*}
\frac{\overline\G_{C}(D+CY)}{\sqrt{\det(D+CY)}}S_{C,D+CY}(\tau)
&=\frac{\theta(\delta\alpha\tau)}{\theta(\tau)}\\
&=\frac{\theta(\delta\alpha\tau)}{\theta(\alpha\tau)}\\
&=\frac{\theta(\delta(\tau+Y))}{\theta(\tau+Y)}\\
&=\frac{\overline\G_C(D)}{\sqrt{\det D}}S_{C,D}(\tau+Y).
\end{align*}

(b)  We know that
$$\frac{S_{C,D}(\tau)}{S_{EC,ED}(\tau)}$$
is an analytic function whose square is 1, and whose limit as $\tau\to 0$ is also 1.
Hence $S_{C,D}(\tau)=S_{EC,ED}(\tau).$

(c)  Since $(ED)^{-1}\,(EC)=D^{-1}C$
and $\Z^{1,n}E=\Z^{1,n}$, we have $\G_{EC}(ED)=\G_C(D)$.  
Also, we have 
\begin{align*}
\overline\G_{CE}(D\,^tE^{-1})
&=\sum_{U\in\Z^{1,n}/\Z^{1,n}D\,^tE^{-1}}\e\{2\,^tUU\,^tED^{-1}CE\}\\
&=\sum_{U\in\Z^{1,n}/\Z^{1,n}D\,^tE^{-1}}\e\{2(E\,^tU)(U\,^tE)D^{-1}C\}
\end{align*}
(recall that $Tr(AB)=Tr(BA)$.)
Take $U'=U\,^tE$.  Thus $U'$ varies over $\Z^{1,n}\,^tE/\Z^{1,n}D
=\Z^{1,n}/\Z^{1,n}D$ as $U$ varies over $\Z^{1,n}/\Z^{1,n}D\,^tE^{-1}$.  So
$\G_{CE}(D\,^tE^{-1})=\G_C(D).$
\end{proof}

We will make use of the following terminology and notation from the theory of quadratic forms.
With $\F$ a field, $V$ an $m$-dimensional $\F$-vector space equipped with a quadratic form $Q$, and $A\in\F^{m,m}_{\sym}$, we write $V\simeq A$ when $A$ represents the quadratic form $Q$ on $V$ relative to some basis for $V$.  With $V\simeq A$, we say $V$ is regular if $\det A\not=0$.  
For a vector $v\in V$, we say $v$ is isotropic if $Q(v)=0$, and anisotropic otherwise.
For $A, A'$ square, symmetric matrices, we sometimes write $A\perp A'$ for the matrix $\diag\{A,A'\}$.  With $a_1,\ldots,a_r\in\F$, we write
$\big<a_1,\ldots,a_r\big>$ for $\diag\{a_1,\ldots,a_r\}$.  We write
$\big<a\big>^{\ell}$ to denote the $\ell\times\ell$ matrix $\diag\{a,\ldots,a\}$.

For a prime $q$, our formulas for the action of the Hecke operators on Eisenstein series will involve the functions we now define.  
% Fix a prime $q$ and let $\varepsilon=\varepsilon_q=\left(\frac{-1}{q}\right)$.  
For $b,c\in\Z$ with $0\le c\le b$, set
$$\bmu_q(b,c)=\prod_{i=0}^{c-1}(q^{b-i}-1),\ 
\bdelta_q(b,c)=\prod_{i=0}^{c-1}(q^{b-i}+1),$$
and set
$$\bbeta_q(b,c)=\frac{\bmu_q(b,c)}{\bmu_q(c,c)}.$$
(So $\bbeta_q(b,c)$ is the number of $c$-dimensional subspaces of a $b$-dimensional space over $\Z/q\Z$.)
We agree that $\bbeta_q(0,0)=1$ and with $0\le b<c$, $\bbeta_q(b,c)=0$.
Now fix a character $\chi$ whose conductor is exactly divisible by $q$.  Then define
$$\sym_q^{\chi}(b,c)=\sum_U \chi_q\left(\det\begin{pmatrix}\mu&\nu\\ ^t\nu&0\end{pmatrix}\right)$$
where $\F=\Z/q\Z$ and
 $\begin{pmatrix}\mu&\nu\\ ^t\nu&0\end{pmatrix}\in\F^{b+c,b+c}_{\sym}$ with $\mu$ of size $b\times b$.  We agree that $\sym_q^{\chi}(0,0)=1$, and we set $\sym_q^{\chi}(b)=\sym_q^{\chi}(b,0)$.

\begin{lem}  Let $q$ be an odd prime,  and
suppose that $\chi$ is a character whose conductor is exactly divisible by $q$.
Then
$$\sym_q^{\chi}(b,c)=\begin{cases}
\frac{q^{m^2+m-c}\bmu(b,b)}{\bmu\bdelta(m-c,m-c)}&\text{if $b+c=2m$ and $\chi_q=1$,}\\
\frac{\varepsilon^mq^{m^2}\bmu(b,b)}{\bmu\bdelta(m-c,m-c)}&\text{if $b+c=2m$, $\chi_q^2=1$, and $\chi_q\not=1$,}\\
\frac{q^{m^2+m}\bmu(b,b)}{\bmu\bdelta(m-c,m-c)}&\text{if $b+c=2m+1$ and $\chi_q=1$,}\\
0&\text{otherwise.}
\end{cases}$$
In particular, $\sym_q^{\chi}(b,c)=0$ unless $\chi_q^2=1$.
\end{lem}

\begin{proof}  
To help us compute $\sym_q^{\chi}(b,c)$, for $\alpha\in\F^{\times}$ we let 
$\sym_q(b,c;\alpha)$ denote the number of $U=\begin{pmatrix} \mu&\nu\\^t\nu&0\end{pmatrix}\in\F_{\sym}^{b+c,b+c}$
with $\mu$ $b\times b$ and $\det U=\alpha$.  

Set $r=b+c$.  With $V$ an $r$-dimensional vector space over $\F$,
an invertible matrix $A\in\F^{r,r}_{\sym}$ defines a regular quadratic form $Q$ on $V$.
Since $q$ is odd, by Theorem 2.11 \cite{Ger} $V$ has a diagonal basis (relative to the quadratic form $Q$).  Then by Proposition 2.51 and Theorem 2.52 of \cite{Ger},
we have that $V\simeq I$ or $V\simeq I\perp\big<\omega\big>$
where $\omega$ is a fixed, non-square element of $\F^{\times}$.
If we change the basis for $V$ by a matrix $G\in GL_m(\F)$, we get $V\simeq \,^tGIG$ or
$V\simeq\,^tG(I\perp\big<\omega\big>)G.$  Thus when $V\simeq I$, any matrix for the quadratic form on
$V$ has determinant $\alpha^2$ for some $\alpha\in\F^{\times}$, and when $V\simeq I\perp\big<\omega\big>,$
any matrix for the quadratic form on $V$ has determinant $\alpha^2\omega$ for some $\alpha\in\F^{\times}$.
Also, note that for $\alpha\in\F^{\times}$,
$$\begin{pmatrix} \mu&\nu\\^t\nu&0\endpmatrix\mapsto\pmatrix I\\&\alpha\end{pmatrix}
\begin{pmatrix} \mu&\nu\\^t\nu&0\endpmatrix\pmatrix I\\&\alpha\end{pmatrix}$$
gives us a bijection between the matrices counted by $\sym(b,c;1)$ and those counted by $\sym(b,c;\alpha^2)$,
and between the matrices counted by $\sym(b,c;\omega)$ and those counted by $\sym(b,c;\omega\alpha^2).$
Thus
%$$\sym_q(b,c)=\frac{q-1}{2}(\sym_q(b,c;1)+\sym_q(b,c;\omega)$$
%and
$$\sym_q^{\chi}(b,c)
=\frac{1}{2}\left(\sym_q(b,c;1)+\chi_q(\omega)\sym_q(b,c;\omega)\right)
\sum_{\alpha\in\F^{\times}}\chi_q(\alpha^2).$$
(The factor of $1/2$ is to account for the value $\alpha^2$ appearing twice as $\alpha$ varies over $\F^{\times}$.)
Thus $\sym_q^{\chi}(b,c)=0$ if $\chi_q^2\not=1$.

On the other hand, we can compute 
$$\sum_{\alpha^2\not=0}\sym_q(b,c;\alpha^2)
\text{ and } \sum_{\alpha^2\not=0}\sym_q(b,c;\omega\alpha^2)$$ as follows. First we choose a basis $y_{b+1},\ldots,y_{b+c}$ for a dimension $c$ subspace of $V$ that is totally isotropic (meaning that the quadratic form on $\F y_{b+1}\oplus\cdots\oplus\F y_{b+c}$ is identically 0).
Then we extend this to a basis $y_1,\ldots,y_{b+c}$ for $V$.  In this way we construct all bases for $V$ relative to which $V\simeq\begin{pmatrix}\mu&\nu\\^t\nu&0\end{pmatrix}$ where $\mu$ is $b\times b$.  There are $o(V)$ bases that yield the same matrix, where $o(V)$ denotes the order of the orthogonal group of $V$.  Using Theorems 2.19, 2.59, 2.60 from \cite{Ger} we compute these quantities to obtain the formulas for $\sym_q^{\chi}(b,c)$.
\end{proof}

\bigskip

\section{Defining Eisenstein series}  

Fix $\stufe\in\Z_+$.
The (degree zero) 
cusps of the Siegel half-space $\h_{(n)}$ under the action of the congruence subgroup $\Gamma_0(4\stufe)$ correspond to the elements of the double quotient
$\Gamma_{\infty}\backslash Sp_n(\Z)\slash\Gamma_0(4\stufe)$ where
$$\Gamma_{\infty}=\left\{\begin{pmatrix}G&GY\\0&^tG^{-1}\end{pmatrix}:\ G\in GL_n(\Z),\ 
Y\in\Z^{n,n}_{\sym}
\ \right\},$$
$$\Gamma_0(4\stufe)=\left\{\begin{pmatrix}A&B\\C&D\end{pmatrix}\in Sp_n(\Z):\ 4\stufe|C\ \right\}.$$
% Note that in \cite{int wt}, we showed that for $\gamma\in Sp_n(\Z)$, there is some 
% $M_0\in\Z^{n,n}_{\sym}$ so that with
% $\gamma_0=\begin{pmatrix}I&0\\M_0&I\end{pmatrix}$ we have
% $\gamma\in\Gamma_{\infty}\gamma_0\Gamma_0(4\stufe).$ 
%so it suffices to only consider such $\gamma_0$.

Given $\gamma\in Sp_n(\Z)$, we want to construct a half-integral weight Eisenstein series generated by the $\Gamma_0(4\stufe)$-orbit of $\Gamma_{\infty}\gamma$ and transforming with some character $\chi$.
We begin by defining an Eisenstein series for the group
 $$\Gamma(4\stufe)=\{\beta\in Sp_n(\Z):\ \beta\equiv I\ (\modulo 4\stufe)\ \},$$
as follows.  With $\delta\in Sp_n(\Z)$ and $\varphi_{\delta}(\tau)$ an automorphy factor for $\delta$ (assuming $\varphi_{\delta}(\tau)=\theta(\delta\tau)/\theta(\tau)$ when $\delta\in\Gamma_0(4)$), we set
$$1(\tau)|[\delta,\varphi_{\delta}(\tau)]=(\varphi_{\delta}(\tau))^{-k}.$$
Then with 
$$\Gamma_{\infty}\Gamma(4\stufe)=\cup_{\delta^*}\Gamma_{\infty}\delta^*
\text{ (disjoint)},$$
we set
$$\E^*(\tau)=\sum_{\delta^*}1(\tau)|\widetilde\delta^*.$$
Since $1(\tau)|\widetilde\beta=1$ for any $\beta\in\Gamma_{\infty}$, the (formal) sum for $\E^*(\tau)$ is well-defined, and provided $n>(k+1)/2$, the sum on $\delta^*$ is absolutely convergent and in fact is analytic (in all the variables of $\tau$).  
Also, $\E^*\not=0$ as $\lim_{\tau\to i\infty}\E^*(\tau)=1$ since $\lim_{\tau\to i\infty}1(\tau)|\widetilde\delta^*=0$ unless $\delta^*\in\Gamma_{\infty}$.
For any $\alpha\in\Gamma(4\stufe)$, we have $\Gamma_{\infty}\Gamma(4\stufe)=\cup_{\delta^*}\Gamma_{\infty}\delta^*\alpha$ (disjoint) and $\widetilde\delta^*\widetilde\alpha=\widetilde{\delta^*\alpha}$, so
$$\E^*(\tau)|\widetilde\alpha=\sum_{\delta^*}1(\tau)|\widetilde{\delta^*\alpha}=\E^*(\tau).$$
Thus $\E^*$ is a (nonzero) Eisenstein series for $\Gamma(4\stufe)$ with weight $k/2$.

Now take $\gamma\in Sp_n(\Z)$ with automorphy factor $\varphi_{\gamma}(\tau)$, and fix a character $\chi$ modulo $4\stufe$.  Assume that $n>(k+1)/2$.  We would like to define an Eisentstein series supported on the $\Gamma_0(4\stufe)$-orbit of $\Gamma_{\infty}\gamma$ by
$$\sum_{\delta}\overline\chi(\delta)\E^*(\tau)|\widetilde\gamma\widetilde\delta$$
where $$\Gamma_{\infty}\gamma\Gamma_0(4\stufe)
=\cup_{\delta}\Gamma_{\infty}\Gamma(4\stufe)\gamma\delta \text{ (disjoint) and }
\chi(\delta)=\chi(\det D_{\delta}).$$
However, this sum is not well-defined unless, for all $\alpha\in\Gamma_0(4\stufe)$ so that
$\Gamma_{\infty}\Gamma(4\stufe)\gamma\alpha=\Gamma_{\infty}\Gamma(4\stufe)\gamma$, we have
$$\overline\chi(\alpha)\E^*(\tau)|\widetilde\gamma\widetilde\alpha=
\E^*(\tau)|\widetilde\gamma.$$
With this in mind, we have the following definitions and lemma.

\smallskip
\noindent{\bf Definition.}
Fix a level $4\stufe$, a character $\chi$ module $4\stufe$, $\gamma\in Sp_n(\Z)$,
and automorphy factor $\varphi_{\gamma}(\tau)$ for $\gamma$ (with $\varphi_{\gamma}(\tau)=\theta(\gamma\tau)/\theta(\tau)$ if $\gamma\in\Gamma_0(4)$).  Set
$$\Gamma_{\gamma}=\left\{\alpha\in\Gamma_0(4\stufe):\ 
\Gamma_{\infty}\Gamma(4\stufe)\gamma\alpha=\Gamma_{\infty}\Gamma(4\stufe)\gamma
\ \right\},$$
and set
$$\Gamma'_{\gamma,\chi}=\left\{\alpha\in\Gamma_{\gamma}:\ 
\overline\chi(\alpha)\E^*(\tau)|\widetilde\gamma\widetilde\alpha
=\E^*(\tau)|\widetilde\gamma\ \right\}.$$
(Note that $\Gamma_{\gamma}$ is a group.)
We give an alternative definition of $\Gamma'_{\gamma,\chi}$ as follows.
For $\alpha\in\Gamma_{\gamma}$, we have $\gamma\alpha\gamma^{-1}\in\Gamma_{\infty}\Gamma(4\stufe)$ and hence
$\E^*(\tau)|\widetilde{\gamma\alpha\gamma^{-1}}=\E^*(\tau).$
Thus for $\alpha\in\Gamma_{\gamma}$, we have $\alpha\in\Gamma'_{\gamma,\chi}$ if and only if
$$\overline\chi(\alpha)\E^*(\tau)|\widetilde\gamma\widetilde\alpha\widetilde\gamma^{-1}
=\E^*(\tau)|\widetilde{\gamma\alpha\gamma^{-1}}.$$
Here $\varphi_{\gamma^{-1}}(\tau)=\frac{1}{\varphi_{\gamma}(\gamma^{-1}\tau)}$
so that $\widetilde\gamma\widetilde\gamma^{-1}=\widetilde I.$
So defining $\zeta_{\gamma}(\alpha):\Gamma_{\gamma}\to\C^{\times}$ by the relation
$$\widetilde\gamma\widetilde\alpha\widetilde\gamma^{-1}(\widetilde{\gamma\alpha^{-1}\gamma^{-1}})
=[I,\zeta_{\gamma}(\alpha)]$$ we have
$$\Gamma'_{\gamma,\chi}=\left\{\alpha\in\Gamma_{\gamma}:\ 
\chi\zeta_{\gamma}^k(\alpha)=1\ \right\}.$$
\smallskip

We now establish some basic properties about $\zeta_{\gamma}$.

\begin{lem} Fix a level $4\stufe$,  
 $\gamma\in Sp_n(\Z)$, and $\varphi_{_\gamma}(\tau)$ an automorphy factor for $\gamma$, with $\varphi_{_\gamma}(\tau)=\theta(\gamma\tau)/\theta(\tau)$ if $\gamma\in\Gamma_0(4).$  
The map $\zeta_{\gamma}:\Gamma_{\gamma}\to\C^{\times}$ is a homomorphism taking values in the multiplicative group $\{\pm1, \pm i\}$.
Thus $\chi\zeta^k_{\gamma}:\Gamma_{\gamma}\to\C^{\times}$ is a homomorphism with finite image and kernel
$\Gamma'_{\gamma,\chi}$.
If $\gamma\in\Gamma_0(4)$ and $\alpha\in\Gamma_{\gamma}$, we have $\zeta_{\gamma}(\alpha)=1.$  
\end{lem}

\begin{proof}
Take $\alpha\in\Gamma_{\gamma}$; note that we have
$\alpha,\gamma\alpha\gamma^{-1}\in\Gamma_0(4)$.
When $\gamma\in \Gamma_0(4)$, we have
$$\widetilde\gamma\widetilde\alpha\widetilde\gamma^{-1}=\widetilde{(\gamma\alpha\gamma^{-1})}$$
and hence $\zeta_{\gamma}(\alpha)=1$.  

Now suppose that $\gamma\not\in\Gamma_0(4).$
For $\delta=\begin{pmatrix}A&B\\C&D\end{pmatrix}\in Sp_n(\Z)$, let $\psi_{\delta}(\tau)=\det(C\tau+D)$.  One easily checks that for $\delta'\in Sp_n(\Z)$, we have
$\psi_{\delta}(\delta'\tau)\psi_{\delta'}(\tau)=\psi_{\delta\delta'}(\tau).$
Also, for any $\delta\in Sp_n(\Z)$ with automorphy factor $\varphi_{_\delta}(\tau)$the function $(\varphi_{_\delta}(\tau))^2/\psi_{\delta}(\tau)$ is  analytic  with absolute value 1, thus for some $w(\delta)$ with $|w(\delta)|=1$, we have
 $(\varphi_{_\delta}(\tau))^2/\psi_{\delta}(\tau)=w(\delta)$.
Then we have
$$w(\gamma^{-1})\psi_{\gamma^{-1}}(\tau)
=\frac{1}{w(\gamma)\psi_{\gamma}(\gamma^{-1}\tau)}$$
and so
$$\frac{1}{w(\gamma)}=w(\gamma^{-1})\psi_{\gamma}(\gamma^{-1}\tau)\psi_{\gamma^{-1}}(\tau)=w(\gamma^{-1}).$$
So
\begin{align*}
\zeta_{\gamma}(\alpha)^2
&=w(\gamma)\,
\psi_{\gamma}(\alpha\gamma^{-1}\cdot\gamma\alpha^{-1}\gamma^{-1}\tau)
\,w(\alpha)\,
\psi_{\alpha}(\gamma^{-1}\cdot\gamma\alpha^{-1}\gamma^{-1}\tau)\\
&\quad\cdot\frac{\psi_{\gamma^{-1}}(\gamma\alpha^{-1}\gamma^{-1}\tau)}{w(\gamma)}
\,w(\gamma\alpha^{-1}\gamma^{-1})\,\psi_{\gamma\alpha^{-1}\gamma^{-1}}(\tau)\\
&=w(\alpha)\,w(\gamma\alpha^{-1}\gamma^{-1}).
\end{align*}
Since
$\alpha\in\Gamma_{\gamma}\subseteq\Gamma_0(4\stufe)$, we have $\gamma\alpha^{-1}\gamma^{-1}\in\Gamma_{\infty}\Gamma(4\stufe)\subseteq\Gamma_0(4\stufe)$, and hence by 
the Transformation Formula we know that
$w(\alpha), w(\gamma\alpha^{-1}\gamma^{-1})$ are 
squares of normalised Gauss sums.  Thus we have $\zeta_{\gamma}(\alpha)\in\{\pm1,\pm i\}.$

To see that $\zeta_{\gamma}$ is a homomorphism, take $\alpha,\delta\in\Gamma_{\gamma}$.
Then, since we have $\alpha,\delta,\gamma\alpha^{-1}\gamma^{-1},\gamma\delta^{-1}\gamma^{-1}\in\Gamma_0(4)$,
\begin{align*}
[I,\zeta_{\gamma}(\alpha\delta)]
&=\widetilde\gamma\widetilde{\alpha\delta}\widetilde\gamma^{-1}
\widetilde{(\gamma\delta^{-1}\alpha^{-1}\gamma^{-1})}\\
&=(\widetilde\gamma\widetilde\alpha\widetilde\gamma^{-1})(\widetilde\gamma\widetilde\delta\widetilde\gamma^{-1})
(\widetilde{\gamma\delta^{-1}\gamma^{-1}})
(\widetilde{\gamma\alpha^{-1}\gamma^{-1}})\\
&=
(\widetilde\gamma\widetilde\alpha\widetilde\gamma^{-1})
[I,\zeta_{\gamma}(\delta)]
(\widetilde{\gamma\alpha^{-1}\gamma^{-1}})\\
&=
(\widetilde\gamma\widetilde\alpha\widetilde\gamma^{-1})
(\widetilde{\gamma\alpha^{-1}\gamma^{-1}})[I,\zeta_{\gamma}(\delta)]\\
&= [I,\zeta_{\alpha}(\delta)][I,\zeta_{\gamma}(\delta)]\\
&=[I,\zeta_{\gamma}(\alpha)\zeta_{\gamma}(\delta)].
\end{align*}
Hence $\zeta_{\gamma}$ is a homomorphism and $\chi\zeta^k:\Gamma_{\gamma}\to\C^{\times}$ is a homomorphism with finite image.  So $\Gamma'_{\gamma,\chi}$ is a normal subgroup of finite index in $\Gamma_{\gamma}$.
\end{proof}

Now fix $\gamma\in Sp_n(\Z)$ and fix an automorphy factor $\varphi_{_\gamma}(\tau)$ for $\gamma$;
assume that $n>(k+1)/2$.
Set
$$\E_{\gamma}=\frac{1}{[\Gamma_{\gamma}:\Gamma'_{\gamma,\chi}]} 
\sum \overline\chi(\delta'\delta)\,\E^*|\widetilde\gamma\widetilde\delta'\widetilde\delta$$
where 
$$\Gamma_0(4\stufe)=\cup_{\delta}\Gamma_{\gamma}\delta \text{ (disjoint),}\ 
\Gamma_{\gamma}=\cup_{\delta'}\Gamma'_{\gamma,\chi}\delta' \text{ (disjoint).}$$
Hence
$$\Gamma_0(4\stufe)=\cup_{\delta',\delta}\Gamma'_{\gamma,\chi}\delta'\delta
\text{ (disjoint).}$$
Since  $\Gamma'_{\gamma,\chi}$ has finite index in $\Gamma_{\gamma}$,  $\Gamma(4\stufe)\subseteq\Gamma_{\gamma}$,
and $\Gamma(4\stufe)$ has finite index in $\Gamma_0(4\stufe)$, we have
$\E_{\gamma}$ defined as a finite sum.
To see that $\E_{\gamma}$ is well-defined, take $\beta\in\Gamma'_{\gamma,\chi}$.  
Thus
\begin{align*}
\overline\chi(\beta)\E^*|\widetilde\gamma\widetilde\beta
&=\overline\chi(\beta)\E^*|[I,\zeta_{\gamma}(\beta)]\widetilde{\gamma\beta\gamma^{-1}}|\widetilde\gamma\\
&=\E^*|\widetilde{\gamma\beta\gamma^{-1}}|\widetilde\gamma\\
&=\E^*|\widetilde\gamma
\end{align*}
since $\gamma\beta\gamma^{-1}\in\Gamma_{\infty}\Gamma(4\stufe).$
Thus $\E_{\gamma}$ is well-defined.
With $\delta,\delta'$ varying as above, by Lemma 3.1 we have
\begin{align*}
[\Gamma_{\gamma}:\Gamma'_{\gamma,\chi}]\E_{\gamma}
&=\sum_{\delta'\delta}\overline\chi(\delta'\delta)\E^*|\widetilde\gamma\widetilde{\delta'\delta}\\
&=\sum_{\delta}\overline\chi(\delta)\left(\sum_{\delta'}\overline\chi\zeta_{\gamma}^{-k}(\delta')\right)\E^*|\widetilde\gamma\widetilde\delta.
\end{align*}
Since $\chi\zeta_{\gamma}^k$ is a homomorphism on $\Gamma_{\gamma}$, we have that $\E_{\gamma}=0$ unless
$\chi\zeta_{\gamma}^k$ is trivial on $\Gamma_{\gamma}$ (meaning that $\Gamma'_{\gamma,\chi}=\Gamma_{\gamma}$).
Note that for any $\alpha\in\Gamma_0(4\stufe)$, we have
$$\Gamma_{\infty}\gamma\Gamma_0(4\stufe)
=\cup_{\delta}\Gamma_{\infty}\Gamma(4\stufe)\gamma\delta\alpha \text{ (disjoint),}$$
and since $\widetilde\delta\widetilde\alpha=\widetilde{\delta\alpha}$ for any $\delta$ as above, we get
$$\E_{\gamma}|\widetilde\alpha
=\sum_{\delta}\overline\chi(\delta)\E^*|\widetilde\gamma\widetilde{\delta\alpha}
=\chi(\alpha)\E_{\gamma}.$$
Hence
$\E_{\gamma}$ is a weight $k/2$, level $4\stufe$ Eisenstein series with character $\chi$.

Henceforth, we assume that $n, k, \stufe\in\Z_+$ are fixed with $k$ odd and $n>(k+1)/2$, and we fix a character $\chi$ modulo $4\stufe$.

\begin{prop}  Take $\gamma=\begin{pmatrix}*&*\\M_0&N_0\end{pmatrix}\in Sp_n(\Z)$, $\varphi_{_\gamma}(\tau)$ an automorphy factor for $\gamma$, and $\chi$ a character modulo $4\stufe$.  Suppose that $\Gamma'_{\gamma,\chi}=\Gamma_{\gamma}$.  Then $\E_{\gamma}\not=0$ and
$$\E_{\gamma}(\tau)=\sum_{\beta}\overline\chi(\beta)1(\tau)|\widetilde\gamma\widetilde\beta$$
where $\Gamma_{\infty}\gamma\Gamma_0(4\stufe)=\cup_{\beta}\Gamma_{\infty}\gamma\beta$ (disjoint).
If $\gamma\in\Gamma_0(4)$ then
$$\E_{\gamma}(\tau)=\sum_{(M\,N)}\overline\chi(M,N)\left(\frac{\overline\G_M(N)}{\sqrt{\det N}}S_{M,N}(\tau)\right)^{-k}$$
where $GL_n(\Z)(M\ N)$ varies over $GL_n(\Z)(M_0\ N_0)\Gamma_0(4\stufe)$ and
$\chi(M,N)=\chi(\beta)$ where $\beta\in\Gamma_0(4\stufe)$ so that $GL_n(\Z)(M\ N)=GL_n(\Z)(M_0\ N_0)\beta.$
\end{prop}

\begin{proof} 
With $\delta^*\in\Gamma(4\stufe)$ and $\delta\in\Gamma_0(4\stufe)$ so that $\Gamma_{\infty}\Gamma(4\stufe)=\cup_{\delta^*}\Gamma_{\infty}\delta^*$ (disjoint),
$\Gamma_0(4\stufe)=\cup_{\delta}\Gamma_{\gamma}\delta$ (disjoint), one easily sees that $\Gamma_{\infty}\gamma\Gamma_0(4\stufe)=\cup_{\delta^*,\delta}\Gamma_{\infty}\delta^*\gamma\delta$ (disjoint). 
Also, $\gamma^{-1}\delta^*\gamma\in\Gamma(4\stufe)$, so
$$\Gamma_{\infty}\gamma\Gamma_0(4\stufe)=\cup_{\delta^*,\delta}\Gamma_{\infty}\gamma(\gamma^{-1}\delta^*\gamma)\delta
\text{ (disjoint)}.$$
Since $\gamma^{-1}\delta^*\gamma\in\Gamma_{\gamma}=\Gamma'_{\gamma,\chi}$,
we have
$$1(\tau)|\widetilde\delta^*\widetilde\gamma
=1(\tau)|\widetilde\gamma|\widetilde\gamma^{-1}\widetilde\delta^*\widetilde\gamma
=1(\tau)|\widetilde\gamma|\widetilde{\gamma^{-1}\delta^*\gamma}.$$
Hence
$$\E_{\gamma}(\tau)=\sum_{\beta}\overline\chi(\beta)1(\tau)|\widetilde\gamma\widetilde\beta$$
where $\Gamma_{\infty}\gamma\Gamma_0(4\stufe)=\cup_{\beta}\Gamma_{\infty}\gamma\beta$ (disjoint).
% One easily checks that with $\delta\in\Gamma_0(4\stufe)$ varying so that 
% $\Gamma_0(4\stufe)=\cup_{\delta}\Gamma_{\gamma}\delta$ (disjoint), we have
% $$\Gamma_{\infty}\gamma\Gamma_0(4\stufe)
% =\cup_{\delta}\Gamma_{\infty}\Gamma(4\stufe)\gamma\delta \text{ (disjoint).}$$
Also, for $\delta^*\in\Gamma_{\infty}\Gamma(4\stufe)$ and $\delta\in\Gamma_0(4\stufe)$, we have
$$\lim_{\tau\to i\infty}1(\tau)|\widetilde\delta^*\widetilde\gamma\widetilde\delta\widetilde\gamma^{-1}=0$$
unless $\delta^*\gamma\delta\gamma^{-1}\in\Gamma_{\infty}$, in which case $\gamma\delta\gamma^{-1}\in\Gamma_{\infty}\Gamma(4\stufe)$ and so $\delta\in\Gamma_{\gamma}$.  Thus
$$\lim_{\tau\to i\infty}\E_{\gamma}(\tau)|\widetilde\gamma^{-1}
=\lim_{\tau\to i\infty}\E^*(\tau)=1,$$
so $\E_{\gamma}\not=0$.

Now suppose that $\gamma\in\Gamma_0(4)$.
If $\beta\in\begin{pmatrix}*&*\\M&N\end{pmatrix}
\in\Gamma_{\infty}\gamma\Gamma_0(4\stufe)$ 
then $(M\ N)\in GL_n(\Z)(M_0\ N_0)\Gamma_0(4\stufe)$.  On the other hand, suppose $(M\ N)$ is a coprime symmetric pair.  Thus there is some 
$\alpha=\begin{pmatrix}*&*\\M&N\end{pmatrix}
\in\Gamma_{\infty}\gamma\Gamma_0(4\stufe),$
and if $(M\ N)\in GL_n(\Z)(M_0\ N_0)\delta$ for some $\delta\in\Gamma_0(4\stufe)$, then one easily checks that $\alpha\in\Gamma_{\infty}\gamma\delta$.
Further, with $\alpha'=\begin{pmatrix}*&*\\M'&N'\end{pmatrix}$, we have
$$GL_n(\Z)(M\ N)=GL_n(\Z)(M'\ N')=GL_n(\Z)(M_0\ N_0)\delta$$
if and only if 
$\Gamma_{\infty}\alpha=\Gamma_{\infty}\alpha'=\Gamma_{\infty}\gamma\delta.$
Thus 
$$\E_{\gamma}(\tau)=\sum_{(M\,N)}\overline\chi(M,N)\left(\frac{\overline\G_M(N)}{\sqrt{\det N}}S_{M,N}(\tau)\right)^{-k}$$
where $GL_(\Z)(M\ N)$ varies (once) over $GL_n(\Z)(M_0\ N_0)\Gamma_0(4\stufe).$
\end{proof}

Next we describe a basis for the space of Eisenstein series.

\begin{prop}  
%Fix $k,\stufe\in\Z_+$ with $k$ odd, and fix a character $\chi$ modulo $4\stufe$.  
Let $\gamma$ vary over elements of $Sp_n(\Z)$ so that
$$Sp_n(\Z)=\cup_{\gamma}\Gamma_{\infty}\gamma\Gamma_0(4\stufe).$$
Then the corresponding nonzero weight $k/2$, level $4\stufe$ Eisenstein series with character $\chi$  are linearly independent.
Further, suppose that $\gamma\in Sp_n(\Z)$,
$\beta\in\Gamma_0(4\stufe)$, $\beta^*\in\Gamma_{\infty}\Gamma(4\stufe)$ so that $\alpha=\beta^*\gamma\beta$; set $\widetilde\alpha=\widetilde\beta^*\widetilde\gamma\widetilde\beta$.  Then $\E_{\alpha}=\chi(\beta)\E_{\gamma}.$
\end{prop}

\begin{proof}  
Suppose $\gamma,\beta\in Sp_n(\Z)$ so that $\beta\not\in\Gamma_{\infty}\gamma\Gamma_0(4\stufe)$ and $\E_{\gamma},\E_{\beta}\not=0$.  As shown in the proof of Proposition 3.2,
$$\lim_{\tau\to i\infty}\E_{\gamma}|\widetilde\gamma^{-1}=1,$$
but as
$$\lim_{\tau\to i\infty}1(\tau)|\widetilde\delta^*\widetilde\beta\widetilde\delta\widetilde\gamma^{-1}=0$$ for all $\delta^*\in\Gamma_{\infty}\Gamma(4\stufe)$ and $\delta\in\Gamma_0(4\stufe)$, we have
$$\lim_{\tau\to i\infty}\E_{\beta}(\tau)|\widetilde\gamma^{-1}=0.$$
Consequently $\E_{\gamma}$ is linearly independent of the set
$$\{\E_{\beta}:\ \beta\in Sp_n(\Z),\ \beta\not\in\Gamma_{\infty}\gamma\Gamma_0(4\stufe)\ \}.$$

Now take $\gamma\in Sp_n(\Z)$, $\beta^*\in\Gamma_{\infty}$, $\beta\in\Gamma_0(4\stufe)$.
Set $\widetilde\alpha=\widetilde\beta^*\widetilde\gamma\widetilde\beta$.  With $\delta'$ verying so that 
$$\Gamma_{\infty}\alpha\Gamma_0(4\stufe)=\cup_{\delta'}\Gamma_{\infty}\alpha\delta'
\text{ (disjoint),}$$
we have
$$\Gamma_{\infty}\gamma\Gamma_0(4\stufe)
=\cup_{\delta'}\Gamma_{\infty}\gamma(\beta\delta') \text{ (disjoint).}$$
Hence
$$\E_{\alpha}=\sum_{\delta'}\overline\chi(\delta')\E^*|\widetilde\alpha\widetilde\delta'
=\sum_{\delta'}\overline\chi(\delta')\E^*|\widetilde\gamma\widetilde{(\beta\delta')}
=\chi(\beta)\E_{\gamma}$$
(recall that for $\beta,\delta'\in\Gamma_0(4)$, we have $\widetilde\beta\widetilde\delta'
=\widetilde{\beta\delta'}$).
\end{proof}

In this paper we are particularly interested in Eisenstein series of level $4\stufe$ where $\stufe$ is odd and square-free.  Below we introduce some terminology and then exhibit a set of representatives for the degree zero cusps.

\smallskip\noindent{\bf Definition.}  Suppose $\stufe\in\Z_+$ is odd and square-free.
We say 
$$\sigma=((\stufe_0,\ldots,\stufe_n),(d,d',\epsilon))$$
is an admissible type for level $4\stufe$ if 
$(\stufe_0,\ldots,\stufe_n)$ is a multiplicative partition of $\stufe$, $d,d'\in\Z_{\ge 0}$ so that $d+d'\le n$, and $\epsilon=+$ or $-$ if $2|d'$ with $d'\not=0$, $\epsilon=+$ otherwise.  
We say $M\in\Z^{n,n}_{\sym}$ is of $\sigma$-type if
$$M\equiv I_s\perp\big<0\big>^{n-s}\ (\stufe_s)$$
for all $s$, $0\le s\le n$,
and 
$$M\equiv 
\begin{cases} I_d\perp 2I_{d'}\perp\big<0\big>^{n-d-d'}\ (4)&\text{if $\epsilon=+$,}\\
I_d\perp2\begin{pmatrix}0&1\\1&0\end{pmatrix}^{d'/2}\perp\big<0\big>^{n-d-d'}
\ (4)
&\text{if $\epsilon=-$.}
\end{cases}$$

\begin{prop}  Suppose $\stufe\in\Z_+$ is odd and square-free.  
For each admissible type $\sigma$ for level $4\stufe$, fix $M_{\sigma}$ of $\sigma$-type.  Then
then with $\sigma$ varying over all admissible types,
$$\left\{\begin{pmatrix}I&0\\M_{\sigma}&I\end{pmatrix}\right\}_{\sigma}$$ is a set of representatives for
$\Gamma_{\infty}\backslash Sp_n(\Z)/\Gamma_0(4\stufe).$
\end{prop}

\begin{proof}  Fix $\gamma=\begin{pmatrix}*&*\\M&N\end{pmatrix}\in Sp_n(\Z)$.  By Proposition 35 \cite{int wt} we know there is some $M'\in\Z^{n,n}_{\sym}$ so that
for all primes $q|4\stufe$,
$$M'\equiv I_{r(q)}\perp\big<0\big>^{n-r(q)}\ (q)$$
where $r(q)=\rank_qM$.  Let $(\stufe_0,\ldots,\stufe_n)$ be the multiplicative partition of $\stufe$ so that for each prime $q|\stufe$, $q|\stufe_s$ if and only if $r(q)=s$.

Set $d=r(2)$.
By \S63 \cite{O'M}, we know there is some $E'\in SL_n(\Z_2)$ so that 
% $E'M'\,^tE'$ is an orthogonal sum of unary and binary components.  So 
$$E'M'\,^tE'=J_0\perp 2J_1\perp 4J$$
where $J_0$, $J_1$ are unimodular over $\Z_2$, $J_0$ is $d\times d$, and $$J_1=I_{d''}\perp\begin{pmatrix}a_1&b_1\\b_1&c_1\end{pmatrix}
\perp\cdots\perp\begin{pmatrix}a_r&b_r\\b_r&c_r\end{pmatrix}$$
with $a_ic_i-b_i^2\not\equiv0\ (2)$ ($1\le i\le r$).
Note that 
$$\begin{pmatrix}1&1\\0&1\end{pmatrix}\begin{pmatrix}0&1\\1&1\end{pmatrix}
\begin{pmatrix}1&0\\1&1\end{pmatrix}\equiv I\equiv 
\begin{pmatrix}1&0\\1&1\end{pmatrix}\begin{pmatrix}1&1\\1&0\end{pmatrix}\begin{pmatrix}1&1\\0&1\end{pmatrix}
\ (2).$$
% Thus we can conjugate $J_1$ by an invertible matrix to assume that $J_1=I_{d'}\perp 2\begin{pmatrix}0&1\\1&0\end{pmatrix}^{\ell}.$
So adjusting $E'$, we can assume that $J_1\equiv I_{\ell}\perp\begin{pmatrix}0&1\\1&0\end{pmatrix}^m\ (2)$ for some $\ell,m\in\Z_{\ge 0}$.
We also have
$$\begin{pmatrix}1&1&1\\1&0&1\\1&1&0\end{pmatrix}
\begin{pmatrix}1&0&0\\0&0&1\\0&1&0\end{pmatrix}
\begin{pmatrix}1&1&1\\1&0&1\\1&1&0\end{pmatrix}
\equiv I\ (2).$$
Hence, further adjusting $E'$, we can assume that 
$$J_1\equiv I_{d'} \text{ or }\begin{pmatrix}0&1\\1&0\end{pmatrix}^{d'/2}\ (2)$$
where $J_1$ is $d'\times d'$.

Now (using Lemma 6.1 \cite{int wt}) choose $E\in SL_n(\Z)$ so that $E\equiv I\ (\stufe)$ and
$E\equiv E'\ (4)$.  Thus
$$\begin{pmatrix}I&0\\ M''&I\end{pmatrix}
=\begin{pmatrix}^tE^{-1}\\&E\end{pmatrix}
\begin{pmatrix}I&0\\M'&I\end{pmatrix}
\begin{pmatrix}^tE\\&E^{-1}\end{pmatrix}
\in \Gamma_{\infty}\gamma\Gamma_0(4\stufe)$$
with $M''\equiv M'\ (\stufe)$, $M''\equiv J_0\perp 2J_1\perp\big<0\big>^{n-d-d'}\ (4)$
where $J_0, J_1$ are symmetric and invertible modulo 2, and 
$J_1\equiv I_{d'} \text{ or }\begin{pmatrix}0&1\\1&0\end{pmatrix}^{d'/2}\ (2).$
Now take $\overline J_0\in\Z^{n,n}_{\sym}$ so that $J_0\overline J_0\equiv I\ (4).$
Take $\delta\in \Gamma_0(4\stufe)$ so that $\delta\equiv I\ (\stufe)$ and
$$\delta\equiv \begin{pmatrix}\overline J_0&&\overline J_0-I\\&I_{n-d}&&0\\
&&J_0\\&&&I_{n-d}\end{pmatrix}\ (4).$$
Set $\epsilon=+$ if $J_1\equiv I_{d'}\ (2)$, and set $\epsilon=-$ otherwise.
Then
$$\begin{pmatrix}I&0\\M''&I\end{pmatrix}\delta\equiv\begin{pmatrix}I&0\\M_{\sigma}&I\end{pmatrix}\ (4\stufe)$$
where $\sigma=((\stufe_0,\ldots,\stufe_n),(d,d',\epsilon))$.  Hence by Proposition 3.3  
\cite{int wt},
 $$\gamma\in\Gamma_{\infty}
\begin{pmatrix}I&0\\M_{\sigma}&I\end{pmatrix}\Gamma_0(4\stufe).$$
Thus 
$\cup_{\sigma}\Gamma_{\infty}\begin{pmatrix}I&0\\M_{\sigma}&I\end{pmatrix}=Sp_n(\Z).$

Now we want to show the above union is disjoint.
So suppose that $M_{\sigma}$ and $M_{\sigma'}$ are 
$\sigma$- and $\sigma'$-type (respectively) where
$\sigma((\stufe_0,\ldots,\stufe_n),(d,d',\epsilon))$ and $\sigma'=((\stufe_0',\ldots,\stufe_n'),(r,r',\epsilon'))$, and suppose that
$$\begin{pmatrix}I&0\\M_{\sigma'}&I\end{pmatrix}
\in\Gamma_{\infty}\begin{pmatrix}I&0\\M_{\sigma}&I\end{pmatrix}\Gamma_0(4\stufe).$$
Thus
$$\begin{pmatrix}I&0\\M_{\sigma'}&I\end{pmatrix}
=\begin{pmatrix}^tE^{-1}&*\\0&E\end{pmatrix}
\begin{pmatrix}I&0\\M_{\sigma}&I\end{pmatrix}
\begin{pmatrix}A&B\\C&D\end{pmatrix}$$
for some $E \in GL_n(\Z)$ and $\begin{pmatrix}A&B\\C&D\end{pmatrix}\in \Gamma_0(4\stufe)$.
So for all primes $q|4\stufe$, $\rank_qM_{\sigma'}=\rank_qM_{\sigma}$.
This means that $\stufe_i'=\stufe_i$ for $0\le i\le n$ and $r=d$.
Write
$$E=\begin{pmatrix}E_1&E_2\\E_3&E_4\end{pmatrix},\ 
A=\begin{pmatrix}A_1&A_2\\A_3&A_4\end{pmatrix}$$
where $E_1, A_1$ are $d\times d$.
We have
$$M_{\sigma}\equiv I_d\perp 2J\ (4)
\text{ and\ } M_{\sigma'}\equiv I_d\perp 2J'\ (4)$$
where $J=J_1\perp\big<0\big>^{n-d-d'}$, 
$J'=J_1'\perp\big<0\big>^{n-d-r'}$, 
$J_1$ is $d'\times d'$, $J_1'$ is $r'\times r'$, and $J_1, J_1'$ are invertible
modulo 2.  
We have $M_{\sigma'}\equiv EM_{\sigma}A\ (4)$, so
$E_1, A_1$ are invertible modulo 2, $E_3, A_2\equiv 0\ (2)$, $E_4, A_4$ are invertible modulo 2, and $J'\equiv E_4JA_4\ (2).$  Hence $\rank_2 J'=\rank_2 J$, meaning that $r'=d'$.
Writing $D=\begin{pmatrix}D_1&D_2\\D_3&D_4\end{pmatrix}$ where $D_1$ is $d\times d$, and knowing that
$E(M_{\sigma}B+D)=I$, we see that we must have $D_3\equiv 0\ (2)$ and $E_4D_4\equiv I\ (2).$  We also know that $A\,^tD\equiv I\ (4\stufe)$, so $A_4\,^tD_4\equiv I\ (2)$.  Thus $E_4\equiv\,^tA_4\ (2).$
Hence $J'\equiv\ ^tA_4JA_4\ (2).$
Now suppose that $d'\ge 1$ and write $A_4=\begin{pmatrix}\alpha_1&\alpha_2\\\alpha_3&\alpha_4\end{pmatrix}$ where $\alpha_1$ is $d'\times d'$.  So
$$\begin{pmatrix}J_1'\\&0\end{pmatrix}\equiv
\begin{pmatrix}^t\alpha_1J_1\alpha_1&^t\alpha_1J_1\alpha_2\\^t\alpha_2J_1\alpha_1&^t\alpha_2J_1\alpha_2\end{pmatrix}\ (2),$$
and hence $\alpha_1$ is invertible modulo 2, $\alpha_2\equiv0\ (2)$.  Thus over $\Z/2\Z$, $J_1$ and $J_1'$ represent the same quadratic form.  Noting that $I_{d'}$ has an anisotropic vector modulo 2 and for any $\ell\ge 1$,
$\begin{pmatrix}0&1\\1&0\end{pmatrix}^{\ell}$ does not, we must have $\epsilon'=\epsilon$.  Thus $\sigma'=\sigma$ whenever $\begin{pmatrix}I&0\\M_{\sigma'}&I\end{pmatrix}\in 
\Gamma_{\infty}\begin{pmatrix}I&0\\M_{\sigma}&I\end{pmatrix}\Gamma_0(4\stufe).$
\end{proof}

Next, for $\gamma\in\Gamma_0(4)$, we determine necessary conditions on $\chi$ to have $\E_{\gamma}\not=0$.  Then for certain $\gamma\not\in\Gamma_0(4)$, we show that $\E_{\gamma}=0$ regardless of choices for $\chi$ and $\varphi_{\gamma}(\tau).$

\begin{prop}  Suppose that $\stufe\in\Z_+$ is odd and square-free, and suppose that $\sigma
=((\stufe_0,\ldots,\stufe_n),(0,0,+))$ is an admissible type for level $4\stufe$.  Take $M_{\sigma}$ of $\sigma$-type, and set $\gamma=\begin{pmatrix}I&0\\M_{\sigma}&I\end{pmatrix}$.  Thus $\gamma\in\Gamma_0(4)$ and $\E_{\gamma}=0$ unless $\chi_q^2=1$ for all primes $q$ dividing $\stufe_1\cdots\stufe_{n-1}$.  (Note that we necessarily have $\chi_4^2=1.$)
\end{prop}  

\begin{proof}  Take $q$ to be a prime dividing $\stufe_1\cdots\stufe_{n-1}$.  Hence $M_{\sigma}\equiv\begin{pmatrix}I_d\\&0\end{pmatrix}\ (q)$ where $0<d<n$.  Let $u\in\Z$ be a unit modulo $q$, with $\overline u\in\Z$ so that $u\overline u\equiv1\ (q)$.  
By Lemma 6.1 \cite{int wt}, we can take
$E\in SL_n(\Z)$ and $\delta\in Sp_n(\Z)$ so that $E\equiv I\ (4\stufe/q)$, $E\equiv\begin{pmatrix}u\\&I\\&&\overline u\end{pmatrix}\ (q)$, $\delta\equiv I\ (4\stufe/q)$,
$$\delta\equiv\begin{pmatrix}\overline u&&&w\\&I_{n-2}&&&0\\&&\overline u&&&0\\&&&u\\&&&&I_{n-2}\\&&&&&u\end{pmatrix}
\ (q)$$ where $w=\overline u-u$.  Thus $\delta\in\Gamma_0(4\stufe)$, $\beta=\begin{pmatrix}^tE^{-1}\\&E\end{pmatrix}\in\Gamma_{\infty}$, and
$\beta\gamma\delta\gamma^{-1}\in\Gamma_{\infty}\Gamma(4\stufe)$.  Hence $\delta\in\Gamma_{\gamma}$.  Since $\gamma\in\Gamma_0(4)$, we have that $\zeta_{\gamma}(\delta)=1$.  So by Lemma 3.1, if $\E_{\gamma}\not=0$ then $1=\chi(\delta)=\chi_q^2(u)$.  This argument holds for all units modulo $q$, and for all primes $q$ dividing $\stufe_1\cdots\stufe_{n-1}$, proving the proposition.
\end{proof}

\begin{prop}
Suppose $\stufe\in\Z_+$ is odd and square-free, and $\sigma=((\stufe_0,\ldots,\stufe_n),(d,d',+))$ is an admissible type for level $4\stufe$ with $d'>0$.
Then with $\chi$ a character modulo $4\stufe$, $M_{\sigma}$ of $\sigma$-type, and $\gamma=\begin{pmatrix}1&0\\M_{\sigma}&1\end{pmatrix}$, we have $\E_{\gamma}=0,$
regardless of the choice of automorphy factor for $\gamma$.
\end{prop}

\begin{proof}  Fix a choice of character $\chi$ modulo $4\stufe$ and an automorphy factor $\varphi_{\gamma}(\tau)$ for $\gamma$.
We have $M_{\sigma}\equiv\big<m_1,\ldots,m_n\big>\ (4\stufe)$ where 
$$m_{d+1}\equiv
\begin{cases} 0\ (\stufe_0\cdots\stufe_d),\\
1\ (\stufe_{d+1}\cdots\stufe_n),\\
2\ (4).
\end{cases}$$
%$m_{\ell}\equiv1\ (\stufe_{\ell}\cdots\stufe_n)$,
%$m_{\ell}\equiv0\ (\stufe_0\cdots\stufe_{\ell-1}),$ and
%$$m_{\ell}\equiv\begin{cases} 1\ (4)&\text{if $\ell\le d$,}\\2\ (4)&\text{if $d<\ell\le %d+d'$,}\\0\ (4)&\text{otherwise.}\end{cases}$$
So (using Proposition 3.3) we can assume that
$M_{\sigma}=\big<m_1,\ldots,m_n\big>.$

Now set $m=m_{d+1},$
\begin{align*}
A&=\begin{pmatrix}I_d\\&1+2\stufe\\&&I_{n-d-1}\end{pmatrix},\ 
B=\begin{pmatrix}0_d\\&2\stufe/m\\&&0_{n-d-1}\end{pmatrix},\\
C&=\begin{pmatrix}0_d\\&-2\stufe m\\&&0_{n-d-1}\end{pmatrix},\ 
D=\begin{pmatrix}I_d\\&1-2\stufe\\&&I_{n-d-1}\end{pmatrix}.
\end{align*}
One easily checks that with 
$$\alpha=\begin{pmatrix}A&B\\C&D\end{pmatrix},$$
we have $(M_{\sigma}\ I)\alpha=(M_{\sigma}\ I)$ and consequently $\alpha\in\Gamma_{\gamma}$.  As we saw in our construction of Eisenstein series, we have $\E_{\gamma}=0$ unless $\Gamma'_{\gamma,\chi}=\Gamma_{\gamma}$.
Then as in the proof of Lemma 3.1, we have
$$(\zeta_{\gamma}(\alpha))^2=\frac{(\overline\G_{-2\stufe m}(1-2\stufe))^2}{1-2\stufe}
=\frac{(\overline\G_{2\stufe m}(2\stufe-1))^2}{1-2\stufe}.$$
Since $\stufe$ is odd, we have $2\stufe-1\equiv 1\ (4)$ and hence
$(\overline\G_{2\stufe m}(2\stufe-1))^2=2\stufe-1$.
Therefore $(\zeta_{\gamma}(\alpha))^2=-1$ and so $\zeta_{\gamma}(\alpha)=\pm i$.
We also have 
$\chi(\alpha)=\chi(1-2\stufe)=\chi_4(-1)=\pm 1.$
Consequently, $\chi\zeta_{\gamma}^k(\alpha)=\pm i$.  Hence $\Gamma'_{\gamma,\chi}\not=\Gamma_{\gamma}$, and so $\E_{\gamma}=0$.
\end{proof}

Using elements of the above proof, we prove the following.

\begin{cor}  Let $\widetilde\Gamma_0(4)=\left\{\widetilde\gamma=[\gamma,\theta(\gamma\tau)/\theta(\tau)]:\ \gamma\in\Gamma_0(4)\ \right\}.$
Suppose that $\Gamma$ is a subgroup of $Sp_n(\Z)$ containing $\Gamma_0(4)$, and that $\widetilde\Gamma$ is a cover of $\Gamma$, meaning that $\widetilde\Gamma$ is a group whose elements are of the form
$[\gamma,\varphi_{\gamma}(\tau)]$, $\gamma\in\Gamma$ and $\varphi_{\gamma}(\tau)$ an automorphy factor for $\gamma$. 
Suppose further that $\widetilde\Gamma_0(4)\subseteq\widetilde\Gamma$ and that
$\begin{pmatrix}I&0\\M&I\end{pmatrix}\in\Gamma$ with
$M\equiv I_d\perp 2I_{d'}\perp\big<0\big>^{n-d-d'}\ (4)$ where $d+d'>0$,
or $M\equiv I_d\perp \begin{pmatrix}0&2\\2&0\end{pmatrix}^{d'/2}\perp\big<0\big>^{n-d-d'}\ (4)$ where $d>0$.
Then $\widetilde\Gamma$ is not a one-fold cover of $\Gamma$.  In particular, when $n=1$, there is no group $\Gamma$ so that $\Gamma_0(4)\subsetneq\Gamma\subseteq SL_2(\Z)$ and $\widetilde\Gamma$ is a one-fold cover of $\Gamma$.
\end{cor}

\begin{proof}
With $M$ as above, we have 
$$\begin{pmatrix}I&0\\M&I\end{pmatrix},\begin{pmatrix}I&0\\M&I\end{pmatrix}^2\in\Gamma;$$ 
consequently $\Gamma$ contains a matrix $\gamma=\begin{pmatrix}I&0\\M'&I\end{pmatrix}$ where $M'\equiv 2I_{r}\perp\big<0\big>^{n-r}\ (4)$ for some $r>0$.
Take an automorphy factor $\varphi_{\gamma}(\tau)$ for $\gamma$ so that $[\gamma,\varphi_{\gamma}(\tau)]\in\widetilde\Gamma$.  Then since $\widetilde\Gamma$ is a group with identity $[I,1]$, we have
$$\widetilde\gamma^{-1}=[\gamma^{-1},1/\varphi_{\gamma}(\gamma^{-1}\tau)]\in
\widetilde\Gamma.$$
Then from the proof of Proposition 3.7, there is some $\alpha\in\Gamma_{\gamma}\subseteq\Gamma_0(4)$ so that
$$\widetilde\gamma\widetilde\alpha\widetilde\gamma^{-1}
=[I,\zeta_{\gamma}(\alpha)]\widetilde{(\gamma\alpha\gamma^{-1})}
\not=\widetilde{\gamma\alpha\gamma^{-1}}.$$
Thus there are (at least) two distinct elements in $\widetilde\Gamma$ of the form
$[\gamma\alpha\gamma^{-1},*]$.  Hence $\widetilde\Gamma$ is not a one-fold cover of $\Gamma$.
\end{proof}

\smallskip\noindent
{\bf Remark.} It would be interesting to know whether there is a group $\Gamma$ with $\Gamma_0(4)\subsetneq\Gamma\subseteq Sp_n(\Z)$ so that $\widetilde\Gamma$ is a one-fold cover of $\Gamma$.  Note that by the above result and Proposition 3.5, such a group $\Gamma$ would necessarily have an element of the form
$$\begin{pmatrix}I&0\\M'&I\end{pmatrix}$$
where $M'\equiv \begin{pmatrix}0&2\\2&0\end{pmatrix}^r\perp\big<0\big>^{n-r/2}\ (4)$
where $r>0$.

\bigskip
\section{The action of Hecke operators attached to odd primes where the level  is $4\stufe$ with $\stufe$ odd and square-free}

Throughout, we fix $n,k,\stufe\in\Z_+$ with $k$ odd, $n>(k+1)/2$, $\stufe$ odd and square-free, and $\chi$ a character modulo $4\stufe$.

Here we look at the action on Siegel Eisenstein series of Hecke operators.  However, because of the constraints reflected in Proposition 2.2, we are only able to do this satisfactorily for Hecke operators attached to odd primes, and for Eisenstein series supported on the $\Gamma_0(4\stufe)$-orbit of a matrix $\gamma\in\Gamma_0(4).$  (In the process of our evaluation, we point out where we find these restrictions necessary.)  We restrict our attention to level $4\stufe$ with $\stufe$ odd and square-free so that we can evaluate the action of the Hecke operators attached to primes dividing $\stufe$.  

In Proposition 3.4 we presented a set of representatives for the double quotient $\Gamma_{\infty}\backslash Sp_n(\Z)/\Gamma_0(4\stufe)$.  From that we can see that a representatives for $\Gamma_{\infty}\backslash\Gamma_0(4)/\Gamma_0(4\stufe)$ are associated to admissible types $((\stufe_0,\ldots,\stufe_n),(0,0,+))$ where $(\stufe_0,\ldots,\stufe_n)$ is a (multiplicative) partition of $\stufe$.  Hence to ease our notation, for $\sigma=(\stufe_0,\ldots,\stufe_n)$ a partition of $\stufe$, we fix a diagonal matrix $M_{\sigma}\in4\Z^{n,n}$ with
$M_{\sigma}\equiv I_s\perp\big<0\big>^{n-s}\ (\stufe_s)$ for all $s$, $0\le s\le n$.
Then we set $\gamma_{\sigma}=\begin{pmatrix}I&0\\M_{\sigma}&I\end{pmatrix}$, and we often write $\E_{\sigma}$ for $\E_{\gamma_{\sigma}}$.
Note that with such $\sigma$ and $M_{\sigma}$, we have that 
$$GL_n(\Z)(M_{\sigma}\ I)\Gamma_0(4\stufe)=SL_n(\Z)(M_{\sigma}\ I)\Gamma_0(4\stufe),$$  so we can write $\E_{\sigma}$ as a sum over representatives 
$$SL_n(\Z)(M\ N)\in SL_n(\Z)(M_{\sigma}\ I)\Gamma_0(4\stufe).$$
This observation is useful in allowing us to more easily adapt results from \cite{int wt}, as with the following result.

\begin{prop}
Fix a multiplicative partition $\sigma$ of $\stufe$, and
suppose $\E_{\gamma_{\sigma}}\not=0.$
Take $(M\ N)\in SL_n(\Z)(M_{\sigma}\ I)\gamma$ where
$\gamma\in\Gamma_0(\stufe)$, and fix a prime $q|\stufe$.
There are $E_0,E_1\in SL_n(\Z)$ so that
$$E_0ME_1\equiv\begin{pmatrix} M_1&0\\0&0\end{pmatrix}\ (q)$$ 
with $M_1$ invertible modulo $q$; 
for any such $E_0,E_1$ we have
$$E_0N\,^tE_1^{-1}\equiv\begin{pmatrix} N_1&N_2\\0&N_4\end{pmatrix}\ (q)$$ 
and 
$\chi_q(\gamma)=\chi_q(\det\overline M_1\cdot\det N_4).$
Also, $\chi_4(M,N)=\chi_4(\det N)$.
Further, for any $G\in GL_n(\Z)$, we have
$$\chi(GM,GN)=\chi(\det G)\chi(M,N)=\chi(MG,N\,^tG^{-1}).$$
\end{prop}

\begin{proof}  Essentially, this is Proposition 3.7 from \cite{int wt}, although for that proposition the level is $\stufe$ with $\stufe$ square-free.  So here the difference is that our level is $4\stufe$ with $\stufe$ odd and square-free.  Hence for $\gamma\in\Gamma_0(4\stufe)$ so that
$(M\ N)\in GL_n(\Z)(M_{\sigma}\ I)\gamma$, we have $\chi_4(\gamma)=\chi_4(\det N)$.
\end{proof}

We begin by evaluating $\E_{\sigma}|T_j(q^2)$ ($1\le j\le n$) for $\sigma$ a partition of $\stufe$ and $q$ a prime dividing $\stufe$.  We show that 
the span of these $\E_{\sigma}$ has a basis 
$$\{\widetilde\E_{\sigma}:\ \sigma \text{ a partition of }\stufe\ \}$$
so that each $\widetilde\E_{\sigma}$ is a simultaneous eigenform for
$\{T_n(q^2):\ \text{prime }q|\stufe\ \}.$  Further, we show that for $\sigma,\rho$ distinct partitions of $\stufe$, there is some prime $q|\stufe$ so that the $T_n(q^2)$-eigenvalues of $\widetilde\E_{\sigma}$ and $\widetilde\E_{\rho}$ differ.  Then, since the Hecke operators commute, we can show that each $\widetilde\E_{\sigma}$ is an eigenform for all Hecke operators associated to odd primes, and we explicitly compute the eigenvalues of all $T_j(p^2)$ where $1\le j\le n$ and $p$ an odd primes.

From Proposition 1.3 and Theorem 2.3 of \cite{half-int}, and Proposition 2.1 of \cite{int wt}, we have the following.
Note that here we have  normalized the Hecke operator presented in \cite{half-int}; also, we have simplified a Gauss sum that appeared there.

\begin{prop}  Suppose
$F$ is a modular form of degree $n$, weight $k/2$, level $4\stufe$ and character $\chi$.
\begin{enumerate}
\item[(a)]  For $q$ a prime dividing $4\stufe$ and $1\le j\le n$,we have
$$F|T_j(q^2)=q^{j(k/2-n-1)}\sum_{G,Y}
F|\left[\begin{pmatrix}X_j^{-1}G^{-1}&X_j^{-1}Y\,^tG\\&X_j\,^tG\end{pmatrix},
q^{j/2}\right]$$
where 
$G$ varies over $\K_j(q)$, and $Y$ varies over $\Y_{j}(q^2)$, meaning that $Y=\begin{pmatrix}U&V\\^tV&0\end{pmatrix}$
so that $U$ varies over integral symmetric $j\times j$ matrices modulo $q^2$, and 
$V$ varies over integral $j\times(n-j)$ matrices modulo $q$.
\item[(b)]  For $p$ a prime not dividing $4\stufe$ and $1\le j\le n$, we have
\begin{align*}
&F|T_j(p^2)\\
&=p^{j(k/2-n-1)}
\sum_{\substack{n_0,n_2\\G,Y}}F|\chi(p^{j-n_0+n_2})\\
&\cdot
\left[
\begin{pmatrix}X_{n_0,n_2}^{-1}G^{-1}&X_{n_0,n_2}^{-1}Y\,^tG\\
0&X_{n_0,n_2}\,^tG\end{pmatrix}, 
p^{-j/2+n_0}\G_{Y_1}(pI_{j-n_0-n_2})
\right].
\end{align*}
Here, $n_0,n_2\in\Z_{\ge0}$ vary so that $n_0+n_2\le j$. For each pair $n_0,n_2$, we have
$G=G_1G_2$, where $G_1$ varies over $SL_n(\Z)/\K_{n_0,n_2}(p)$,
$$G_2=\begin{pmatrix} I_{n_0}\\&G'\\&&I_{n_2}\end{pmatrix}$$ with $G'$ varying over
$SL_{n'}(\Z)/\,^t\K'_{j'}(p)$ where $n'=n-n_0-n_2$, $j'=j-n_0-n_2$,
$$\K'_{j'}=\begin{pmatrix} pI_{j'}\\&I\end{pmatrix} SL_{n'}(\Z)\begin{pmatrix}\frac{1}{p}I_{j'}\\&I\end{pmatrix}\cap SL_{n'}(\Z),$$
 and $Y$ varies over $\Y_{n_0,n_2}(p^2)$, the set
of all integral, symmetric $n\times n$ matrices 
$$\begin{pmatrix} Y_0&Y_2&Y_3&0\\^tY_2&Y_1/p&0\\^tY_3&0\\0\end{pmatrix}$$
with $Y_0$ $n_0\times n_0$, varying modulo $p^2$, $Y_1$ $j'\times j'$, varying
modulo $p$ provided $p\nmid\det Y_1$, 
and $Y_2,Y_3$  varying modulo $p$ with  $Y_3$ $n_0\times n_2$.
Note that we can assume $G\equiv I\ (4\stufe)$ and $Y\equiv 0\ (4\stufe)$.
Also,
$$\G_{Y_1}(pI_{j-n_0-n_2})=\left(\frac{\det Y_1}{p}\right)\G_1(p)^{j-n_0-n_2}.$$
\end{enumerate}
\end{prop}

\begin{thm}  Suppose that $\stufe\in\Z_+$ is odd and square-free, $\chi$ is a character modulo $4\stufe$ so that $\chi(-1)=1$.  Fix a prime $q|\stufe$ and a multiplicative partition
$\sigma'=(\stufe'_0,\ldots,\stufe'_n)$ of $\stufe/q$.  For $0\le d\le n$, let
$\sigma_d=(\stufe_0,\ldots,\stufe_n)$ where
$$\stufe_i=\begin{cases}\stufe_i'&\text{if $i\not=d$,}\\
q\stufe_d'&\text{if $i=d$.}
\end{cases}$$
Then when $\E_{\sigma_d}\not=0$ for $0\le j\le n$, we have
$$\E_{\sigma_d}|T_j(q^2)=\sum_{t=0}^{n-d} A_j(d,t)\E_{\sigma_{d+t}}$$
where
\begin{align*}
A_j(d,t)&=q^{(j-t)d-t(t+1)/2}\bbeta_q(d+t,t)\\
&\quad \cdot \sum_{s=0}^j\sum_{d_5=0}^{j-s}\sum_{d_8=0}^{d_5}
q^{a_j(d;s,d_5,d_8)}
\overline\chi_{\stufe/q}(X_{s,r}^{-1}M_{\sigma_{d}}X_j,X_{s,r}^{-1}X_j^{-1})    \\
&\quad \cdot \bbeta_q(d,s)\bbeta_q(t,d_5)\bbeta_q(n-d-t,j-s+d_8-t)\\
&\quad \cdot \bbeta_q(t-d_5,d_8)\sym_q^{\chi}(t-d_5-d_8)\sym_q^{\chi'}(d_5,d_8)\\
&\quad \cdot q^{-k(d_5-d_8)/2}\G_1(q)^{k(d_5-d_8)},\\
%\end{align*}
r&=j-s-d_5+d_8,\\
%\begin{align*}
a_j(d;s,d_5,d_8)
&=(k/2-d)(2s+d_5-d_8)+s(s-d_8-j-1)\\
&\quad +d_8(j-d_5)-d_5(d_5+1)/2+d_8(d_8+1)/2,
\end{align*}
and $\chi'_q(*)=\chi_q(*)\left(\frac{*}{q}\right).$
(Here $\sym_q^{\chi}(b,c)$ is as defined in section 2.)
Thus $\widetilde \E_{\sigma_d}|T_j(q^2)=A_j(d,0)\widetilde\E_{\sigma_d}.$
\end{thm}

\begin{proof}  
Here we make use of results derived in the proof of Theorem 4.4  \cite{int wt}, where we evaluated the action of $T_j(q^2)$ on integral weight Eisenstein series.
The difference is that here we need to compare automorphy factors Gauss sums.

To ease notation, temporarily write $\E_{d'}$
for $\E_{\sigma_{d'}}$, $M_{d'}$ for $M_{\sigma_{d'}}$, $X_{s,r}$ for $X_{s,r}(q)$, $\K_{s,r}$ for $\K_{s,r}(q)$, and
$\Y_{j}$ for $\Y_{j}(q^2)$.

By Proposition 4.2,
we have
\begin{align*}
\E_d(\tau)|T_j(q^2)
&=q^{-j(n+1)} \sum_{(M\,N),G,Y}\overline\chi(M,N)
\left(\frac{\overline\G_M(N)}{\sqrt{\det N}}\right)^{-k}\\
&\qquad\cdot
S_{M,N}(X_j^{-1}G^{-1}(\tau+GY\,^tG)\,^tG^{-1}X_j^{-1})^{-k}
\end{align*}
where $SL_n(\Z)(M\ N)$ varies over the $\Gamma_0(4\stufe)$-orbit of
$SL_n(\Z)(M_d\ I)$, and $G, Y$ vary as in Proposition 4.2.
By Proposition 2.2 we know that
for $E\in SL_n(\Z)$, we have $\chi(EM,EN)=\chi(M,N)$, $\G_{EM}(EN)=\G_M(N)$,
$\sqrt{EN}=\sqrt{N}$, and $S_{EM,EN}(\tau)=S_{M,N}(\tau)$ for all $\tau\in\h_{(n)}$.
Thus for each $G,Y$, we follow \cite{int wt} to adjust the pair $(M\ N)$ by left multiplication from $SL_n(\Z)$ so that, with appropriate choices of $s,r$, we have that
$$(X_{s,r}MX_j^{-1}G^{-1},\,X_{s,r}NX_j\,^tG) = 1.$$
We set
$M'=X_{s,r}MX_j^{-1}G^{-1}$,
$N''=X_{s,r}NX_j\,^tG$,  and $N'=M'GY\,^tG+N''$.
Then
$$\frac{S_{M,N}(X_j^{-1}G^{-1}(\tau+GY\,^tG)\,^tG^{-1}X_j^{-1})}
{S_{M',N''}(\tau+GY\,^tG)}$$
is an analytic function whose square is $q^{r-s-j}$ and whose limit as $\tau\mapsto0$ is 
$q^{(r-s-j)/2}$; thus the quotient above is equal to $q^{(r-s-j)/2}$.  
Further, since $q\not=2$, from Proposition 2.2 we know that 
$$\frac{\overline\G_{M'}(N'')}{\sqrt{\det N''}}\, S_{M',N''}(\tau+GY\,^tG)
=\frac{\overline\G_{M'}(N')}{\sqrt{\det N'}}\,S_{M',N'}(\tau).$$
Hence
\begin{align*}
\E_{d}(\tau)|T_j(q^2)
&=q^{j(k/2-n-1)}\sum_{(M\,N), G, Y}\overline\chi(M,N)
\left(\frac{\overline\G_{M'}(N')}{\sqrt{\det N'}}\right)^{-k}
\\
&\qquad\cdot q^{k(s-r)/2}\left(\frac{\overline\G_M(N)\cdot\sqrt{\det N''}}{\sqrt{\det N}\cdot\overline\G_{M'}(N'')}\right)^{-k}\,S_{M',N'}(\tau)^{-k}
\end{align*}
where $SL_n(\Z)(M\ N)$ varies over $SL_n(\Z)(M_d\ I)\Gamma_0(4\stufe)$,
 $G, Y$ vary as in Proposition 4.2, and $r, s, M', N'', N'$ are as defined above.
Note that we necessarily have $\rank_qM'\ge\rank_qM=d$.

Given a coprime symmetric pair $(M'\ N')$, we want to count how often 
$$SL_n(\Z)(M'\ N')=SL_n(\Z)X_{s,r}(MX_j^{-1}G^{-1}\ \,MX_j^{-1}Y\,^tG+NX_j\,^tG)$$
for some $(M\ N)\in SL_n(\Z)(M_d\ I)\Gamma_0(4\stufe).$  
Equivalently, we want to count how often
\begin{align*}
(*)\qquad& X_{s,r}^{-1}E(M'GX_j\ \,-M'GYX_j^{-1}+N'\,^tG^{-1}X_j^{-1})\\
&\in SL_n(\Z)(M_d\ I)\Gamma_0(4\stufe)
\end{align*}
for some $E\in SL_n(\Z)$.  Since $X_{s,r}^{-1}EX_{s,r}\in SL_n(\Z)$ for $E\in SL_n(\Z)$ if and only if $E\in\K_{s,r}$, we only need to consider $E\in\K_{s,r}\backslash SL_n(\Z)$.
Thus
$$\E_d(\tau)|T_j(q^2)=\sum_{(M'\,N')}c_d(M',N')\overline\chi(M',N')\left(\frac{\overline\G_{M'}(N')}{\sqrt{\det N'}} S_{M',N'}(\tau)\right)^{-k}$$
where
\begin{align*}
c_d(M',N')&=q^{j(k/2-n-1)}\chi(M',N')\sum_{s,r,E,G,Y}\overline\chi(M,N)q^{k(s-r)/2}\\
&\quad\cdot \left(\frac{\overline\G_M(N)}{\sqrt{\det N}}\frac{\sqrt{\det X_{s,r}NX_j}}{\overline\G_{X_{s,r}MX_j^{-1}}(X_{s,r}NX_j)}
\right)^{-k},
\end{align*}
$s,r\in\Z_{\ge0}$, $E\in\K_{s,r}\backslash SL_n(\Z)$, $G\in SL_n(\Z)/\K_j$, $Y\in\Y_j$
such that 
\begin{align*}
(M\ N)&=X_{s,r}^{-1}E(M'GX_j\ \-M'GYX_j^{-1}+N'\,^tG^{-1}X_j^{-1})\\
&\in SL_n(\Z)(M_d\ I)\Gamma_0(4\stufe).
\end{align*}
(Here we have used that for $G\in SL_n(\Z)$ we have $\G_{M'}(N')=\G_{M'G}(N'\,^tG^{-1})$.)
We also know that $\E_d|T_j(q^2)$ is a modular form, and hence is a linear combination of $\E_{d'}$ for $d'\ge d$.
Thus $\E_d|T_j(q^2)=\sum_{d'\ge d}c_d(M_{d'},I)\E_{d'}$.

As shown in the proof of Theorem 4.4 \cite{int wt}, given $s,r$, each solution $E,G,Y$
to (*) corresponds to choices for $s, d_5, d_7, d_8$ so that $s\le d$, $d'=d+d_5+d_7+d_8$ and
$M,N$ have the following forms.
$$M=\begin{pmatrix} A_1'&qA_1&qA_2'&qA_2\\qA_3'&qA_3&A_4'&qA_4\\
qA_5'&qA_5&qA_6'&qA_6\\q^2A_7'&q^2A_7&qA_8'&qA_8\end{pmatrix},\ 
N=\begin{pmatrix}D_1'&D_1&D_2'&D_2\\D_3'&D_3&D_4'&D_4\\
qD_5'&D_5&qD_6'&D_6\\qD_7'&D_7&qD_8'&qD_8\end{pmatrix}$$
where $A_1', D_1'$ are $s\times s$, $A_4', D_4'$ are $(d-s)\times(d-s)$,
$A_7, D_7$ are $r\times(j-s)$,
$A_1', A_4'$ are invertible modulo $q$.  Since $(M,N)=1$, we must have (row)
$\rank_q\begin{pmatrix}D_5&D_6\\D_7&0\end{pmatrix}=n-d$,
and since $(X_{s,r}MX_j^{-1}\ \,X_{s,r}NX_j)=1$, we must have (row)
$\rank_q\begin{pmatrix}A_5&0&0&D_6\\A_7&A_8&D_7&0\end{pmatrix}=n-d.$
So $\rank_qD_7=r$ and $\rank_q(A_5\ D_6)=n-d-r.$
Further, adjusting $E$ by left multiplication from $\K_{s,r}$ and $G$ by right multiplication from
$\K_{j}$, we can assume that modulo $q$,
\begin{align*}
&A_5\equiv\begin{pmatrix}\alpha_5&0&0\\0&0&0\end{pmatrix},\\
&A_7\equiv\begin{pmatrix}0&0&0\\0&0&\alpha_7\\0&0&0\end{pmatrix},\ 
A_8\equiv\begin{pmatrix}0&0\\0&0\\ \alpha_8&0\end{pmatrix}
\end{align*}
where $\alpha_5$ is $d_5\times d_5$,  
$\alpha_7$ is $d_7\times d_7$, $\alpha_8$ is $d_8\times d_8$, and $\alpha_5, \alpha_7, \alpha_8$ are invertible modulo $q$.
(So 
we necessarily have $d_5+d_7\le j-s$ and $d_8\le n-j-d+s$.)
Also, as $(M\ N)$ is a coprime symmetric pair,  modulo $q$ we have
\begin{align*}
&D_5\equiv\begin{pmatrix}\beta_1'&*&*\\0&*&*\end{pmatrix},\ 
D_6\equiv\begin{pmatrix}\gamma_1'&*\\0&\gamma_4\end{pmatrix},\\
&D_7\equiv\begin{pmatrix}0&\delta_2&0\\0&*&\delta_6'\\ \delta_7'&*&*\end{pmatrix}
\end{align*}
where 
$\beta_1'$ is $d_5\times d_5$, $\gamma_1'$ is $d_5\times d_8$, 
$\gamma_4$ is $(n-d-r-d_5)\times(n-j-d+s-d_8)$,
$\delta_2$ is $(r-d_7-d_8)\times(j-s-d_5-d_7)$,
$\delta_6'$ is $d_7\times d_7$,
and $\delta_7'$ is $d_8\times d_5$.
Then a careful analysis in the proof of Theorem 4.4 \cite{int wt} tells us that 
$d_5\ge d_8$, $r=j-s-d_5+d_8$, and $\gamma_4, \delta_2, \delta_6', \begin{pmatrix}\beta_1'&\gamma_1'\\ \delta_7'&0\end{pmatrix}$ are square and invertible modulo $q$.

From the above descriptions of $M,N$, we can see that 
$$(X_{0,r}MX_j^{-1}X_s\ \,X_{0,r}NX_jX_s^{-1})$$
is an integral, coprime pair (which is necessarily symmetric).
Thus by Propositions 2.2(c) and 5.1, we have
$$\G_{X_{s,r}MX_j^{-1}}(X_{s,r}NX_j)  
=q^s \G_{X_{0,r}MX_j^{-1}X_s}(X_{0,r}NX_jX_s^{-1}).$$

Let $$P_1=\begin{pmatrix}0&I_s&0\\I_{j-s}&0&0\\0&0&I_{n-j}\end{pmatrix}.$$
So using Propostion 2.2(c) and recalling that $P_1^{-1}=\,^tP_1$, we have
\begin{align*}
\G_{X_{0,r}MX_j^{-1}X_s}(X_{0,r}NX_jX_s^{-1})
&=\G_{X_{0,r}MX_j^{-1}X_sP_1}(X_{0,r}NX_jX_s^{-1}P_1)\\
&=\G_{X_{0,r}MP_1X_{j-s}^{-1}}(X_{0,r}NP_1X_{j-s}).
\end{align*}

% modification 26/04/16

To prepare to use Proposition 5.2, choose $E_0, G_0\in SL_{d_5}(\Z)$ so that
$$E_0\gamma_1'\equiv\begin{pmatrix}0\\ \gamma_1''\end{pmatrix}\ (q),\ 
\delta_7'G_0\equiv\begin{pmatrix}0& \delta_7''\end{pmatrix}\ (q)$$
where $\gamma_1'', \delta_7''$ are $d_8\times d_8$.  Write
$$E_0\beta_1' G_0=\begin{pmatrix}\beta_1''&\rho_2\\ \rho_3&\rho_4\end{pmatrix}$$
where $\beta_1''$ is $(d_5-d_8)\times(d_5-d_8)$.
So 
$$\begin{pmatrix}E_0\\&I_{d_8}\end{pmatrix}\begin{pmatrix}\beta_1'&\gamma_1'\\ \delta_7'&0\end{pmatrix}
\begin{pmatrix}G_0\\&I_{d_8}\end{pmatrix}
\equiv\begin{pmatrix}\beta_1''&\rho_2&0\\ \rho_3&\rho_4&\gamma_1''\\0&\delta_7''&0\end{pmatrix}\ (q),$$
and since this matrix is invertible modulo $q$, we must have $\rank_q\delta_7''=d_8=\rank_q\gamma_1''$, and $\rank_q\beta''_1=d_5-d_8$.  (So $\beta''_1, \gamma_1'', \delta_7''$ are invertible modulo $q$.)
Write 
$$E_0\alpha_5\,^tG_0^{-1}=\begin{pmatrix}\alpha_5'&\omega_2\\ \omega_3&\omega_4\end{pmatrix}$$
where $\alpha_5'$ is $(d_5-d_8)\times(d_5-d_8)$.  By the symmetry of $M\,^tN$, we have that
$$\begin{pmatrix}E_0\\&I_{d_8}\end{pmatrix}\begin{pmatrix}\beta_1'&\gamma_1'\\ \delta_7'&0\end{pmatrix}
\begin{pmatrix}^t\alpha_5\\&^t\alpha_8\end{pmatrix}\begin{pmatrix}^tE_0\\&I_{d_8}\end{pmatrix}$$
is symmetric modulo $q$; consequently $\omega_2\equiv0\ (q)$ and so $\alpha_5', \omega_4$ are invertible modulo $q$.
Now set
$$E=\begin{pmatrix}I_d\\&E_0\\&&I_{n-d-d_5}\end{pmatrix},\ 
G=\begin{pmatrix}G_0\\&I_{n-d_5}\end{pmatrix},
P_2=\begin{pmatrix}0&I_{d_5-d_8}&0\\I_r&0&0\\0&0&I_{n-j-s}\end{pmatrix}.$$
Then $E$ commutes with $X_{0,r}$ (since $d_8\le n-j-d+s$ and hence $r=j-s-d_5+d_8\le n-d-d_5$).
Somewhat similarly, $G$ and $P_2$ commute with $X_{j-s}$ (since $d_5\le j-s$ and $r+d_5-d_8=j-s$).
Thus by Proposition 2.2(c),
\begin{align*} 
&\G_{X_{0,r}MP_1X_{j-s}^{-1}}(X_{0,r}NP_1X_{j-s})\\
&\quad=
\G_{X_{0,r}EMP_1\,^tG^{-1}P_2X_{j-s}^{-1}}(X_{0,r}ENP_1GP_2X_{j-s}).
\end{align*}
Hence with $\widetilde M=EMP_1\,^tG^{-1}$ and $\widetilde N=ENP_1G$, we know that
\begin{align*}
&(\widetilde MP_2X_{j-s}^{-1}X_r\ \ \widetilde NP_2X_{j-s}X_r^{-1})
\text{ is an integral coprime pair}\\
\iff\ &(\widetilde MP_2X_{j-s}^{-1}X_r\,^tP_2\ \ \widetilde NP_2X_{j-s}X_r^{-1}\,^tP_2)
\text{ is an integral coprime pair}\\
\iff\ &(\widetilde MX_{j-s-r}^{-1}\ \ \widetilde NX_{j-s-r})
\text{ is an integral coprime pair.}
\end{align*}
Recall that $P_1$ has permuted the 1st $j-s$ columns of $M$ with the next $s$ columns of $M$, and similarly for $N$.
(So, for instance, the top row of blocks of $MP_1$ is $(qA_1\ A_1'\ qA_2'\ qA_2)$ and hence $\widetilde MX_{j-s-r}^{-1}$ is integral.)
Using our block decompositions of $M$ and $N$ in terms of subscripted $A$s and $D$s, we have the following.
In $\widetilde MX_{j-s-r}^{-1}$, the $(1,2)$ block is $A_1'$ which is $s\times s$ and invertible modulo $q$, and the $(2,3)$ block is $A_4'$ which is $(d-s)\times(d-s)$ and invertible modulo $q$.  Let $\widetilde A_5$ denote the $(3,1)$ block of $\widetilde MX_{j-s-r}^{-1}$, and $\widetilde D_5, \widetilde D_6$ and  $\widetilde D_y$ the $(3,1), (3,4)$ and $(4,1)$ blocks of $\widetilde NX_{j-s-r}$.  Then modulo $q$, we have
\begin{align*}
\widetilde A_5\equiv\begin{pmatrix}\alpha_5'&0&0&0\\ \omega_3&0&0&0\\ 0&0&0&0\end{pmatrix},\ 
\widetilde D_5&\equiv\begin{pmatrix}q\beta''_1&*&*&*\\0&*&*&*\\0&0&*&*\end{pmatrix},\ 
\widetilde D_6\equiv\begin{pmatrix}0&*\\ \gamma_1''&*\\0&\gamma_4\end{pmatrix},\\
\widetilde D_7&\equiv\begin{pmatrix}0&0&\delta_2&0\\0&0&*&\delta_6'\\0&\delta_7''&*&*\end{pmatrix},
\end{align*}
where (as we've previously noted) $\alpha_5', \beta''_1, \gamma_1'', \gamma_4, \delta_2, \delta_6', \delta_7''$ are square and invertible modulo $q$.
Hence $(\widetilde MX_{j-s-r}^{-1}\ \,\widetilde NX_{j-s-r})$ has $q$-rank $n$, and thus is a coprime symmetric pair.
So by Proposition 5.2 we have
\begin{align*}
\G_{X_{0,r}\widetilde MP_2X_{j-s}^{-1}}(X_{0,r}\widetilde NP_2X_{j-s})
&=\G_{\widetilde MP_2X_{j-s}^{-1}X_r}(\widetilde NP_2X_{j-s}X_r^{-1})\\
&=\G_{\widetilde MX_{j-s-r}^{-1}}(\widetilde NX_{j-s-r})
\end{align*}
since $P_2X_{j-s}X_r^{-1}\,^tP_2=X_{j-s-r}.$
If $j=s+r$ then with Proposition 2.2(c) we have 
$$\G_{X_{0,r}\widetilde MP_2X_{j-s}^{-1}}(X_{0,r}\widetilde NP_2X_{j-s})=\G_{\widetilde M}(\widetilde N)=\G_M(N).$$

Suppose $j>s+r$.  Then we modify $G$ in our previous step to prepare to apply Proposition 5.3.
Take $G_0$ as before, and choose an integral $(d_5-d_8)\times r$ matrix $W$ so that 
$$\begin{pmatrix}E_0\\&I_{n-d-r-d_5}\end{pmatrix}D_5
\begin{pmatrix}G_0\\&I_{j-s-d_5}\end{pmatrix}
\begin{pmatrix}I_{d_5-d_8}&W\\&I_r\end{pmatrix}
\equiv\begin{pmatrix}\beta''_1&0\\ *&*\end{pmatrix}\ (q).$$
Then with 
$$G_1=\begin{pmatrix}G_0\\&I_{j-s-d_5}\end{pmatrix}
\begin{pmatrix}I_{d_5-d_8}&W\\&I_r\end{pmatrix},$$
we have 
$$\begin{pmatrix}E_0\\&I_{n-d-r-d_5}\end{pmatrix}A_5\,^tG_1^{-1}\equiv\begin{pmatrix}\alpha_5'&0\\ *&*\end{pmatrix}\ (q)$$
(recall that $\omega_2\equiv0\ (q)$).  Set $G=\begin{pmatrix}G_1\\&I_{n-j+s}\end{pmatrix}$; so $G$ commutes with $X_{j-s}$.
Hence we again have
$$\G_{X_{0,r}\widetilde MP_2X_{j-s}^{-1}}(X_{0,r}\widetilde NP_2X_{j-s})
=\G_{\widetilde MX_{j-s-r}^{-1}}(\widetilde NX_{j-s-r}).$$
But now, with
$$P_3=\begin{pmatrix}I_d&0&0\\0&0&I_{n-d-d_5+d_8}\\0&I_{d_5-d_8}&0\end{pmatrix},$$
by Proposition 5.3 we have
\begin{align*}
\G_{\widetilde MX_{j-s-r}^{-1}}(\widetilde NX_{j-s-r})
&=\G_{P_3\widetilde MX_{j-s-r}^{-1}}(P_3\widetilde NX_{j-s-r})\\
&=\left(\frac{\det \alpha_5'\beta''_1}{q}\right)\left(\G_1(q)\right)^{d_5-d_8}\G_{\widetilde M}(\widetilde N)
\end{align*}
and we know by Proposition 2.2(c) that $\G_{\widetilde M}(\widetilde N)=\G_M(N).$
Recall that $r=j-s-d_5+d_8$ and $d_5\ge d_8$.  Hence we have $j=s+r$ if and only if $d_5=d_8$.  Hence for all choices of $s,r$, the above computations give us
\begin{align*}
\G_{X_{s,r}MX_j^{-1}}(X_{s,r}NX_j)
%&\quad
=q^s\left(\frac{\det \alpha_5'\beta_1''}{q}\right)\left(\G_1(q)\right)^{d_5-d_8}\G_M(N).
\end{align*}
Now, by the symmetry of $M\,^tN$, we know that $\begin{pmatrix}\beta_1'&\gamma_1'\\ \delta_7'&0\end{pmatrix}
\begin{pmatrix}^t\alpha_5\\&^t\alpha_8\end{pmatrix}$ is symmetric modulo $q$, and hence
$$\left(\frac{\det\begin{pmatrix}\beta_1'&\gamma_1'\\ \delta_7'&0\end{pmatrix}
\begin{pmatrix}^t\alpha_5\\&^t\alpha_8\end{pmatrix}}{q}\right)
=\left(\frac{-1}{q}\right)^{d_5-d_8}\,\left(\frac{\det\alpha_5'\beta_1''}{q}\right).$$
Also, $\overline\G_1(q)=\G_{-1}(q)=\left(\frac{-1}{q}\right)\G_1(q).$  Thus
\begin{align*}
&\frac{\overline\G_M(N)\,\sqrt{\det N''}}{\sqrt{\det N}\,\overline\G_{M'}(N'')}
%&\qquad
=\frac{q^{(d_5-d_8)/2}}{(\G_1(q))^{d_5-d_8}}
\left(\frac{\det
\begin{pmatrix}\beta_1'\,^t\alpha_5&\gamma_1'\,^t\alpha_8\\ \delta_7'\,^t\alpha_5&0
\end{pmatrix}}{q}\right).
\end{align*}
To evaluate $c_{d}(M_{d'},I)$, we also need to evaluate
$$\chi(M,N)=\chi(X_{s,r}^{-1}EM_{d'}GX_j,X_{s,r}E\,^tG^{-1}X_j^{-1}).$$
We note that by Lemma 6.1 \cite{int wt} we can choose $E,G\equiv I\ (4\stufe/q)$, and we can choose
$Y\equiv0\ (4\stufe/q)$.  Thus $M\equiv M_{d'}\ (4\stufe/q)$ so
$$\chi_{4\stufe/q}(M,N)=\chi_{4\stufe/q}(X_{s,r}^{-1}M_{d'}X_j,X_{s,r}^{-1}X_j^{-1}).$$
As shown in the proof of Theorem 4.4 \cite{int wt}, with $(M\ N)$ as above we have
$$\overline\chi_q(M,N)
=\overline\chi_q\left(\det\begin{pmatrix}\beta_1'\,^t\alpha_5&\gamma_1'\,^t\alpha_8\\
\delta_7'\,^t\alpha_5&0\end{pmatrix}\det(\delta_6'\,^t\alpha_7)\right).$$

Now we need to consider what happens when we fix $s,r,d'$ and let $E,G,Y$ vary
so that $X_{s,r}^{-1}E(M_{d'}GX_j\ \ -M_{d'}GYX_j^{-1}+\,^tG^{-1}X_j^{-1})\in SL_n(\Z)(M_{d'}\ I).$
As proved in Theorem 4.4 \cite{int wt}, there are 
\begin{align*}
&\bbeta_q(d,s)\bbeta_q(d'-d,d_5)\bbeta_q(n-d',j-s+d_8-d'+d)\bbeta_q(d'+d-d_5,d_8)\\
&\cdot q^{(d+d_5)(r+d+d_5-d')+s(n-d-d_5)+(d_4+d_8)(j-s-d_5)-d_7d_8}
\end{align*}
permissible choices for $(E,G)$, and for each choice of $(E,G)$, as $Y$ varies over permissible choices,
the matrix 
$$\begin{pmatrix}\beta_1'\,^t\alpha_5&\gamma_1'\,^t\alpha_8&0\\
\delta_7'\,^t\alpha_5&0&0\\
0&0&\delta_6'\,^t\alpha_7\end{pmatrix}$$
varies $q^{(j-s)(n-d+1)-d_5(j-s+d_8+1)-d_7(d_7+1)/2}$ times over all symmetric,
invertible matrices modulo $q$.
As $\beta_1',\gamma_1',\delta_6',\delta_7'$ vary as such,
\begin{align*}
&\sum\overline\chi_q\left(\det\begin{pmatrix}\beta_1'\,^t\alpha_5&\gamma_1'\,^t\alpha_8\\
\delta_7'\,^t\alpha_5&0\end{pmatrix}\det(\delta_6'\,^t\alpha_7)\right)
\left(\frac{\det\begin{pmatrix}\beta_1'\,^t\alpha_5&\gamma_1'\,^t\alpha_8\\
\delta_7'\,^t\alpha_5&0\end{pmatrix}}{q}\right)\\
&\quad = \sym_q^{\overline\chi'}(d_5,d_8)\sym_q^{\overline\chi}(d_7),
\end{align*}
and by Lemma 2.3, this is $\sym_q^{\chi'}(d_5,d_8)\sym_q^{\chi}(d_7)$.
Now combining the above results yields the theorem.
\end{proof}

From this theorem we can deduce a ``multiplicity-one" result.  To ease our description, we introduce the following. 

\smallskip\noindent
{\bf Definition.}  
Let $\sigma,\alpha$ be multiplicative partitions of $\stufe$, and let $q$ be a prime dividing $\stufe$.
We write $\sigma<\alpha\ (q)$ if $\rank_qM_{\sigma}<\rank_qM_{\alpha},$
$\sigma=\alpha\ (q)$ if $\rank_qM_{\sigma}=\rank_qM_{\alpha},$ and
$\sigma\le\alpha\ (q)$ if $\rank_qM_{\sigma}\le\rank_qM_{\alpha}.$
For $Q|\stufe$, we write $\sigma<\alpha\ (Q)$ if $\rank_qM_{\sigma}<\rank_qM_{\alpha}$ for all primes $q|Q$,
$\sigma=\alpha\ (Q)$ if $\rank_qM_{\sigma}=\rank_qM_{\alpha}$ for all primes $q|Q$,
$\sigma\le\alpha\ (Q)$ if $\rank_qM_{\sigma}\le\rank_qM_{\alpha}$ for all primes $q|Q$.

\begin{cor}  Let $\sigma$ be a partition of $\stufe$ 
so that $\E_{\sigma}\not=0$ and let
$q$ a prime dividing $\stufe$; set $d=\rank_qM_{\sigma}$.
For  any partition $\beta$ of $\stufe$ with $\beta\ge\sigma\ (\stufe)$, there are constants $a_{\sigma,\beta}(\stufe)$ so that
$a_{\sigma,\sigma}(\stufe)=1$, and with 
$$\widetilde\E_{\sigma}=\sum_{\beta\ge\sigma\,(\stufe)}a_{\sigma,\beta}(\stufe)\E_{\beta},$$
 we have 
$\widetilde\E_{\sigma}|T_j(q^2)=\lambda_{\sigma;j}(q^2)\widetilde\E_{\sigma}$ where
\begin{align*}
\lambda_{\sigma;j}(q^2)
=&q^{jd}\sum_{s=0}^jq^{s(k-2d+s-j-1)}
\chi_{\stufe'_0}(q^{2s})\chi_{\stufe'_n}(q^{2(j-2)})
\end{align*}
and $\stufe'_i=\stufe_i/(q,\stufe_i)$.
For $\sigma, \rho$ distinct multiplicative partitions of $\stufe$, there is some prime $q|\stufe$ so that $\lambda_{\sigma;n}(q^2)\not=\lambda_{\rho;n}(q^2).$
Further, $\widetilde\E_{\sigma}=0$ if and only if $E_{\sigma}=0$, and 
\begin{align*}
&\spn\{\widetilde\E_{\sigma}:\ \sigma\text{ is a multiplicative partition of }\stufe\ \}\\
&\quad=\spn\{\E_{\sigma}:\ \sigma\text{ is a multiplicative partition of }\stufe\ \}.
\end{align*}
\end{cor}

\begin{proof}  This proof follows the lines of reasoning used to prove Corollaries 4.2 and 4.3 in \cite{int wt}.

First, fix a multiplicative partition $\sigma=(\stufe_0,\ldots,\stufe_n)$ of $\stufe$ and a prime $q|\stufe$. Let $d=\rank_qM_{\sigma}$.  We temporarily use the notaion of Theorem 4.3; so we write $\sigma_d$ for $\sigma$, and for $t>0$, we write $\sigma_{d+t}$ for
$\rho$ where $\rho=\sigma\ (\stufe/q)$ and $\rank_qM_{\rho}=d+t$.
 Then by Theorem 4.3, we have
$$\E_{\sigma_d}|T_n(q^2)=\sum_{t\ge0}A_n(d,t)\E_{\sigma_{d+t}};$$
if $\E_{d+t}=0$ for some $t$ then we can replace $A_n(d,t)$ by 0 in this formula.
The formula for $A_n(d,0)$ is a sum on $s$ with $0\le s\le n$, and the corresponding summand has a term $\bbeta_q(d,s)\bbeta_q(n-d,n-s)$.  Consequently
$$A_n(d,0)=q^{d(k-d-1)}\overline\chi_{\stufe/q}\left(qX^{-1}_{d,n-d}(q),
\frac{1}{q}X^{-1}_{d,n-d}(q)\right)$$
since $\bbeta_q(d,s)=0$ if $s>d$ and $\bbeta_q(n-d,n-s)=0$ if $s<d$.
So we can represent $T_n(q^2)$ on
$^t(\E_{\sigma_0}\,\ldots\,\E_{\sigma_n})$ by an upper triangular matrix, whose $d$th diagonal entry has absolute value $q^{d(k-d-1)}$ when $\E_{\sigma_d}\not=0$; when $\E_{\sigma_d}=0$, we can take the $d$th column of this matrix to be zeros.
Hence we can diagonalize this matrix; correspondingly, for each $\rho$ with $\rho=\sigma\ (\stufe/q)$ and $\E_{\rho}=\not=0$, there are values $a_{\rho,\alpha}(q)$ with $a_{\rho,\rho}(q)=1$ so that 
$$\sum_{\substack{\alpha\ge\rho\,(q)\\ \alpha=\rho\,(\stufe/q)}}a_{\rho,\alpha}(q)\E_{\alpha}$$
is an eigenform for $T_n(q^2)$ with eigenvalue
$\lambda_{\sigma;n}(q^2)$ (as defined in the statement of the corollary).
Further, for $\alpha>\rho\ (q)$, $\alpha=\rho\ (\stufe/q)$,  
by Proposition 2.3 and Theorem 4.3,
we have $a_{\rho,\alpha}(q)=0$ unless $\chi_q^2=1$.

Now, for any prime $q|\stufe$ and $\sigma,\alpha$ multiplicative partitions of $\stufe$ with $\alpha\ge\sigma\ (q)$, set
$$a_{\sigma,\alpha}(q)=a_{\rho,\alpha}(q)$$
where $\rho=\sigma\ (q)$ and $\rho=\alpha\ (\stufe/q)$.  Then
for any $Q|\stufe$ and $\alpha\ge\sigma\ (Q)$, set
$$a_{\sigma,\alpha}(Q)=\prod_{\substack{q|Q\\q\,\text{prime}}}a_{\sigma,\alpha}(q).$$
Set
$$\widetilde\E_{\sigma}=\sum_{\alpha\ge\sigma\,(\stufe)}a_{\sigma,\alpha}(\stufe)\E_{\alpha}.$$
So
\begin{align*}
\widetilde\E_{\sigma}|T_n(q^2)
&=\sum_{\substack{\beta=\sigma\,(q)\\ \beta\ge \sigma\,(\stufe/q)}}a_{\sigma,\beta}(\stufe/q)\sum_{\substack{\alpha\ge \beta\,(q)\\ \alpha=\beta\,(\stufe/q)}}
a_{\beta,\alpha}(q)\E_{\alpha}|T_n(q^2)\\
&=\sum_{\substack{\beta=\sigma\,(q)\\ \beta\ge \sigma\,(\stufe/q)}}
\lambda_{\beta;n}(q^2) a_{\sigma,\beta}(\stufe/q)\sum_{\substack{\alpha\ge \beta\,(q)\\ \alpha=\beta\,(\stufe/q)}}
a_{\beta,\alpha}(q)\E_{\alpha}.
\end{align*}
We claim that for any $\beta$ so that $\beta=\sigma\ (q)$, $\beta\ge\sigma\ (\stufe/q)$,
and $a_{\sigma,\beta}(\stufe/q)\not=0$, we have
$\lambda_{\beta;n}(q^2)=\lambda_{\sigma;n}(q^2)$.  To see this, first note that
since $\beta=\sigma\ (q)$, we have
$$\lambda_{\beta;n}(q^2)
=q^{d(k-d-1)}\overline\chi_{\stufe/q}\left(qX^{-1}_{d,n-d}(q)M_{\beta},
\frac{1}{q}X^{-1}_{d,n-d}\right).$$
Using Proposition 4.1, we see that for a prime $q'|\stufe/q$ and $d'=\rank_{q'}M_{\beta}$, we have
$$\chi_{q'}\left(qX^{-1}_{d,n-d}(q)M_{\beta},
\frac{1}{q}X^{-1}_{d,n-d}\right)=\chi_{q'}^2(q^{d-d'})$$
and by Proposition 3.5, $\chi_{q'}^2=1$ when $q'|\stufe/\stufe_0\stufe_n$ (and necessarily
$\chi_4^2=1$).
As noted above, for $\beta>\sigma\ (q')$ and $\beta\ge\sigma\ (\stufe/q)$, we know that 
$a_{\sigma,\beta}(q')=0$ unless $\chi_{q'}^2=1$.  Thus when $a_{\sigma,\beta}(\stufe/q)\not=0$, we have $\chi_{q'}^2=1$ for all primes $q'|\stufe/q$.  Hence
$$\lambda_{\beta;n}(q^2)=q^{d(k-d-1)}=\lambda_{\sigma;n}(q^2).$$
Consequently $\widetilde\E_{\sigma}|T_n(q^2)=\lambda_{\sigma;n}(q^2)\widetilde\E_{\sigma}.$

Regardless of whether $\chi_{\stufe/q}^2=1$, we have
$|\lambda_{\sigma;n}(q^2)|=q^{d(k-d-1)}$ for every prime $q|\stufe$ and $d=\rank_qM_{\sigma}$.  Hence for $\rho\not=\sigma\ (\stufe)$, there is some prime $q|\stufe$ so that $\rank_qM_{\rho}\not=\rank_qM_{\sigma}$, and hence
$\lambda_{\rho;n}(q^2)\not=\lambda_{\sigma;n}(q^2)$.  This gives us the multiplicity-one result claimed in the statement of the corollary.

Finally, since the Hecke operators commute, we must have that $\widetilde\E_{\sigma}$ is an eigenform for $T_j(q^2)$ for all primes $q|\stufe$ and $1\le j<n$.
Thus using Theorem 4.3, we must have
$$\widetilde\E_{\sigma}|T_j(q^2)=\lambda_{\sigma;j}(q^2)\widetilde\E_{\sigma},$$
as claimed.
\end{proof}

Since the Hecke operators commute, we know that for every odd prime $p\nmid\stufe$ we must have that $\widetilde\E_{\sigma}$ is an eigenform for $T_j(p^2)$; below we compute the eigenvalues.  These are not so attractive, so in the corollary that follows we use an alternate set of generators for the local Hecke algebra, producing much more attractive eigenvalues.

Note that in Corollary 4.4 we have only diagonalized the space of Eisenstein series corresponding to $\Gamma_{\infty}\backslash\Gamma_0(4)/\Gamma_0(4\stufe)$, relative to the Hecke operators $T_j(q^2)$ for primes $q|\stufe$.  Following the proof of Theorem 4.3, we can see that the $\widetilde\E_{\sigma}$ will not all be eigenforms for $T_j(4)$ (unless all the Eisenstein series corresponding to cusps outside $\Gamma_{\infty}\backslash\Gamma_0(4)/\Gamma_0(4\stufe)$ are all 0, which is certainly not the case for Siegel degree $n=1$).

\begin{thm}
Let $\sigma=(\stufe_0,\ldots,\stufe_n)$ be a multiplicative partition of $\stufe$, and suppose that $\E_{\sigma}\not=0$.
Let $p$ be a prime not dividing $4\stufe$, and take $j$ so that $1\le j\le n$.
Then $\E_{\sigma}|T_j(p^2)=\lambda_{j}(p^2)\E_{\sigma}$ where
\begin{align*}
\lambda_{\sigma;j}(p^2)
&= 
\bbeta_p(n,j)\sum_{r+s\le j}p^{k(j-r+s)/2-(j-r)(n+1)}\chi(p^{j-r+s})\chi_{\stufe_n}(p^{2(r-s)})\\
&\quad\cdot \bbeta_p(j,r)\bbeta_p(j-r,s) \left(\frac{\G_1(p)}{\sqrt{p}}\right)^{j-r-s}\sym_p^{\psi}(j-r-s)
\end{align*}
where the sum is over all non-negative integers $r,s$ with $r+s\le j,$
and  $\psi(*)=\left(\frac{*}{p}\right)$.
Further, $\widetilde\E_{\sigma}|T_j(p^2)=\lambda_{\sigma;j}(p^2)\widetilde\E_{\sigma}.$
\end{thm}

\begin{proof}  To a large extent we follow the reasoning of Theorem 5.4 \cite{int wt}.

For any $n_0,n_2\in\Z_{\ge0}$ with $n_0+n_2\le j$, $G\in SL_n(\Z)$, $Y\in\Y_{n_0,n_2}$, and $SL_n(\Z)(M\ N)\in SL_n(\Z)(M_{\sigma}\ I)\Gamma_0(4\stufe)$, we adjust the representative $(M,N)$ and choose $r,s\in\Z_{\ge0}$ so that 
$$(M'\ N'')=X^{-1}_{r,s}(MX_{n_0,n_2}^{-1}G^{-1}\ \ NX_{n_0,n_2}\,^tG)$$
is an integral coprime pair (which is automatically symmetric).  
Note that $M'\equiv 0\ (4)$, and for all primes $q|\stufe$, we have $\rank_qM'=\rank_qM_{\sigma}$.
It follows from Proposition 3.5 that
$(M'\ N'')\in SL_n(\Z)(M_{\sigma}\ I)\Gamma_0(4\stufe)$.
Then as in the proof of Theorem 4.3, we have
\begin{align*}
&S_{M,N}(X_{n_0,n_2}^{-1}G^{-1}(\tau+GY\,^tG)\,^tG^{-1}X_{n_0,n_2}^{-1})\\
&\qquad=p^{(r-s-n_0+n_2)/2}S_{M',N''}(\tau+GY\,^tG)\\
&\qquad=p^{(r-s-n_0+n_2)/2}\frac{\sqrt{\det N''}}{\overline\G_{M'}(N'')}
\frac{\overline\G_{M'}(N')}{\sqrt{\det N'}}S_{M',N'}(\tau)
\end{align*}
where $N'=M'Y+N''$ (and $(M'\ N')\in SL_n(\Z)(M_{\sigma}\ I)\Gamma_0(4\stufe)$).
Setting $\varepsilon=\left(\frac{-1}{p}\right)$, we have
$$\E_{\sigma}|T_j(p^2)
=\sum_{(M'\ N')}c_{\sigma}(M',N')\overline\chi(M',N')
\left(\frac{\overline\G_{M'}(N')}{\sqrt{\det N'}} S_{M',N'}(\tau)\right)^{-k}$$
with
\begin{align*}
c_{\sigma}(M',N')
&=
p^{j(k/2-n-1)}\chi(M',N')\sum\chi(p^{j-n_0+n_2})\overline\chi(M,N)\\
&\cdot\left(\frac{\overline\G_{M}(N)}{\sqrt{\det N}}\frac{\sqrt{\det N''}}{\overline\G_{M'}(N'')}
\right)^{-k} p^{k(s-r)/2}\\
&\cdot
\varepsilon^{(k+1)(j-n_0-n_2)/2}\left(\frac{\det Y_1}{p}\right)
(p^{-1/2}\G_1(p))^{j-n_0-n_2},
\end{align*}
where the sum is over all $r,s,n_0,n_2\in\Z_{\ge0}$ with $n_0+n_2\le j$, $G,Y$ as in Proposition 4.1(b), and $E\in\K_{s,r}\backslash SL_n(\Z)$ so that
\begin{align*}
(M\ N)
&=X_{r,s}E(M'GX_{n_0,n_2}\ \ -M'GYX_{n_0,n_2}^{-1}+N'\,^tG^{-1}X_{n_0,n_2}^{-1})\\
&\in SL_n(\Z)(M_{\sigma'}\ I).
\end{align*}

Now fix a partition $\sigma'$ of $\stufe$.
As noted in \S5 of \cite{int wt}, we can choose the representative $M_{\sigma'}$ to be divisible by $p^3$; then with $(M\ N)$ as above,  we have that $p|M$ and $N$ is invertible modulo $p$.
We also have that $N''=X_{r,s}^{-1}NX_{n_0,n_2}\,^tG$ is invertible modulo $p$; consequently (as proved in Theorem 5.4 \cite{int wt}), we must have $r=n_0$ and $s=n_2$.  
From this we find that 
$$N=\begin{pmatrix}N_1&pN_2&p^2N_3\\N_4&N_5&pN_6\\N_7&N_8&N_9\end{pmatrix}$$
with $N_1$ $r\times r$, $N_9$ $s\times s$, and $N_1,N_5,N_9$ invertible modulo $p$.
Hence $X_r^{-1}NX_r$ is also integral and invertible modulo $p$.
Set $$P=\begin{pmatrix}0&0&I_s\\0&I&0\\I_r&0&0\end{pmatrix}\qquad (n\times n).$$
Then $X_r^{-1}N\,^tPX_{0,r}^{-1}=X_r^{-1}NX_rP$ is invertible modulo $p$, and
using Proposition 2.2(c) and Proposition 5.2, we have

\begin{align*}
\G_{M_{\sigma'}}(N'')
&=\G_{X_{r,s}^{-1}MX_{r,s}^{-1}}(X_{r,s}^{-1}NX_{r,s})\\
&=\G_{X_{0,s}^{-1}(X_r^{-1}M\,^tPX_{0,r})X_s}
(X_{0,s}^{-1}(X_r^{-1}N\,^tPX_{0,r}^{-1})X_s^{-1})\\
&=\G_{X_r^{-1}M\,^tPX_{0,r}}(X_r^{-1}N\,^tPX_{0,r}^{-1})\\
&=\G_{X_{0,r}PMX_r^{-1}}(X_{0,r}PNX_r)\\
&=\G_M(N).
\end{align*}
As shown in the proof of Theorem 4.5 \cite{int wt},
given $r,s$, for all choices of $Y$ and for
$$p^{rs}\bbeta_p(n,j)\bbeta_p(j,r)\bbeta_p(j-r,s)$$ 
choices of $(E,G)$ we have
\begin{align*}
(M\ N)&=
X_{r,s}E(M_{\sigma'}GX_{r,s}\ \ -M_{\sigma'}GYX_{r,s}^{-1}+\,^tG^{-1}X_{r,s}^{-1})\\
&\in SL_n(\Z)(M_{\sigma}\ I)\Gamma_0(4\stufe).
\end{align*}
Writing $Y$ as in Proposition 4.2, we have
$$\sum_{Y}\left(\frac{\det Y_1}{p}\right)
=p^{r(n-s+1)}\sym_p^{\psi}(j-r-s)$$
where $\psi_p(*)=\left(\frac{*}{p}\right).$
Note that $\sym_p^{\psi}(j-r-s)=0$ when $j-r-s$ is odd, and when $j-r-s$ is even we have
$\varepsilon^{(k+1)(j-r-s)/2}=1$.
Also, we can always choose $E,G\equiv I\ (p)$ and $Y\equiv0\ (p)$; hence
$\overline\chi(M,N)=\overline\chi(X_{r,s}M_{\sigma'}X_{r,s},I)$, and by Proposition 4.1,
if $\E_{\sigma'}\not=0$, we have $\overline\chi(M,N)=\chi_{\stufe_n}(p^{2(r-s)}).$
Combining these computations yields the value of $\lambda_{\sigma;j}(p^2).$

Since the Hecke operators commute, by our multiplicity one result (Corollary 4.4), $\widetilde\E_{\sigma}=\sum_{\alpha\ge\sigma\,(\stufe)}a_{\sigma,\alpha}\E_{\alpha}$
must be a $T_j(p^2)$ eigenform.  
Since each $\E_{\alpha}$ is a $T_j(p^2)$ eigenform, we must have
$\lambda_{\alpha;j}(p^2)=\lambda_{\sigma;j}(p^2)$ whenever $a_{\sigma,\alpha}\not=0$
(which can be corroborated by direct computation), so we have
$\widetilde\E_{\sigma}|T_j(p^2)=\lambda_{\sigma;j}(p^2)\widetilde\E_{\sigma}.$
\end{proof}

\begin{cor}
Let $p$ be a prime not dividing $4\stufe$, and set
$\varepsilon=\left(\frac{-1}{p}\right)$.  Set
$$\widetilde T_j(p^2)
=\sum_{\ell=0}^j \chi(p^{j-\ell})\varepsilon^{(k+1)(j-\ell)/2} p^{(j-\ell)(k/2-n-1/2)}
\bbeta_p(n-\ell,j-\ell) T_{\ell}(p^2)$$
and
$$T'_j(p^2)
=\sum_{i=0}^j(-1)^i p^{i(i-1)/2}\bbeta_p(n-j+i,i)\chi_{\stufe_n}(p^{2i})\widetilde T_{j-i}(p^2).$$
With $\sigma=(\stufe_0,\ldots,\stufe_n)$ a multiplicative partition of $\stufe$, we have
$$\E_{\sigma}|T'_j(p^2)=\lambda'_{j}(p^2) \E_{\sigma}
\text{ and } \widetilde\E_{\sigma}|T'_j(p^2)=\lambda'_{\sigma;j}(p^2) \widetilde\E_{\sigma} $$ 
where
$$\lambda'_{\sigma;j}(p^2)
= \bbeta_p(n,j) p^{j(k/2-n-1/2)+j(j-1)/2}\chi'(p^j)
\prod_{i=1}^j(\chi'\overline\chi_{\stufe_n}^2(p)p^{(k+1)/2-i} +1),$$
and $\chi'(p^s)=\chi(p^s)\varepsilon^{s(k+1)/2}$.
\end{cor}

\begin{proof}
To take advantage of a result proved in \cite{half-int}, we set
$$\widetilde\G(\big<0\big>^{\ell})=p^{-\ell}(\G_1(p))^{\ell}\sym_p^{\psi}(\ell)$$
where $\psi(*)=\left(\frac{*}{p}\right)$. 
Then (with $t=\ell-r-s$)
we have $\E_{\sigma}|\widetilde T_j(p^2)=\widetilde\lambda_{j}(p^2) \E_{\sigma}$
where
\begin{align*}
\widetilde\lambda_{\sigma;j}(p^2)
&= \sum_{0\le\ell\le j}\chi(p^{j-\ell})\varepsilon^{(k+1)(j-\ell)/2}p^{(j-\ell)(k/2-n-1/2)}\\
&\quad\cdot\bbeta(n-\ell,j-\ell) \lambda_{\sigma;\ell}(p^2)\\
&= \sum \chi(p^{j-r+s})\chi_{\stufe_n}(p^{2(r-s)})\varepsilon^{(k+1)(j-r-s)}\\
&\quad\cdot p^{(j-r)(k/2-n-1/2)+s(k-1)/2}\\
&\quad\cdot
\bbeta(n-\ell,j-\ell) \bbeta(n,\ell)\bbeta(\ell,r)\bbeta(\ell-r,s)
\, \widetilde\G(\big<0\big>^{\ell-r-s})
\end{align*}
where $0\le \ell\le j$ and $0\le r+s\le \ell$, or equivalently, $0\le r+s\le j$ and $r+s\le \ell\le j$.
We have
\begin{align*}
&\bbeta(n-\ell,j-\ell)\bbeta(n,\ell)\beta(\ell,r)\bbeta(\ell-r,s)\frac{\bmu(j,\ell)}{\bmu(j,\ell)}\\
% &\qquad = \bbeta(n,j)\frac{\bmu(j,r)\bmu(j-r,s)\bmu(j-r-s,\ell-r-s)}
% {\bmu(\ell-r-s,\ell-r-s)\bmu(r,r)\bmu(s,s)}\\
&\qquad = \bbeta(n,j)\bbeta(j,r)\bbeta(j-r,s)\bbeta(j-r-s,\ell-r-s).
\end{align*}
Now we make the change of variables $\ell\mapsto \ell-r-s$.  So
\begin{align*}
\widetilde\lambda_{\sigma;j}(p^2)
&= \sum \chi(p^{j-r+s})\chi_{\stufe_n}(p^{2(r-s)})
\varepsilon^{(k+1)(j-r-s)/2}\\
&\quad\cdot p^{(j-r)(k/2-n-1/2)+s(k-1)/2}\\
&\quad\cdot \bbeta(n,j)\bbeta(j,r)\bbeta(j-r,s)\bbeta(j-r-s,\ell)
\widetilde\G(\big<0\big>^{\ell})
\end{align*}
where $0\le r+s\le j$, $0\le \ell\le j-r-s$. Taking
$W=\big<0\big>^{j-r-s}$ in Lemma 3.2(b) \cite{half-int},
we have that
$$\sum_{\ell=0}^{j-r-s}\bbeta(j-r-s,\ell)\widetilde\G(\big<0\big>^{\ell})
=p^{(j-r-s)(j-r-s-1)/2}$$
(as the form $W\perp\big<2\big>$ primitively represents $\big<0\big>^{j-r-s}$ only once).
Hence
\begin{align*}
\lambda'_{\sigma;j}(p^2)
% &=\sum(-1)^i p^{i(i-1)/2} \bbeta(n-j+i,i) 
% \chi_{\stufe_n}(p^{2i}) \widetilde\lambda_{j-i}(p^2)\\
&=\sum(-1)^ip^{i(i-1)/2}\chi_{\stufe_n}(p^{2(i+r-s)})\chi'(p^{j-i-r+s})\\
&\qquad\cdot
p^{(j-i-r)(k/2-n-1/2)+s(k-1)/2+(j-i-r-s)(j-i-r-s-1)/2}\\
&\qquad\cdot
\bbeta(n-j+i,i)\bbeta(n,j-i)\bbeta(j-i,r)\bbeta(j-i-r,s)
\end{align*}
where $0\le i\le j$ and $0\le r+s\le j-i$, or equivalently,
 $0\le i\le j$, $0\le r\le j-i$, $0\le s\le j-i-r$.
Making the change of variables $r\mapsto j-i-r$, we get
$\lambda'_{j}(p^2)$ as a sum over $0\le i\le j$, $0\le r\le j-i$, $0\le s\le r$,
or equivalently, $0\le r\le j$, $0\le i\le j-r$, $0\le s\le r$.  
We have $\bbeta(j-i,j-i-r)=\bbeta(j-i,r)$ and
$$\bbeta(n-j+i,i)\bbeta(n,j-i)\bbeta(j-i,r) \frac{\bmu(j,i)}{\bmu(j,i)}
=\bbeta(n,j)\bbeta(j,r)\bbeta(j-r,i).$$
Using that $\bbeta(m,r)=p^r\bbeta(m-1,r)+\bbeta(m-1,r-1)$, we find that
$$\sum_{i=0}^{j-r}(-1)^ip^{i(i-1)/2}\bbeta(j-r,i)
=\begin{cases}1&\text{if $r=j$,}\\ 0&\text{otherwise.}\end{cases}$$
Thus
$$\lambda'_{\sigma;j}(p^2)
=\bbeta(n,j)\chi'(p^j) p^{j(k/2-n-1/2)+j(j-1)/2} A(j,(k+1)/2-j)$$
where
$$A(j,y)=\sum_{s=0}^j \phi(p^s)  p^{ys+s(s-1)/2}\bbeta(j,s)$$
and $\phi(p^s)=\chi(p^s)\varepsilon^{(k+1)s/2}\overline\chi_{\stufe_n}(p^{2s})$.
Again using the relation $\bbeta(j,s)=p^s\bbeta(j-1,s)+\bbeta(j-1,s-1)$, we find that
$$A(j,y)=(\phi(p) p^y+1) A(j-1,y+1)
=\prod_{i=0}^{j-1}(\chi'(p)\overline\chi_{\stufe_n}(p^2)p^{y+i}+1).$$
Taking $y=(k+1)/2-j$ shows that $\E_{\sigma}|T'_j(p^2)=\lambda'_{\sigma;j}(p^2)\E_{\sigma}$.

Now recall that $\widetilde\E_{\sigma}=\sum_{\beta\ge\sigma\,(\stufe)}a_{\sigma,\beta}(\stufe)\E_{\beta}.$
Suppose that $\beta\ge\sigma\ (\stufe)$ so that $a_{\sigma,\beta}(\stufe)\not=0$.
Write $\sigma=(\stufe_0,\ldots,\stufe_n)$, $\beta=(\stufe'_0,\ldots,\stufe'_n)$. For any prime $q|\stufe_n$, we know that $\rank_qM_{\beta}\ge\rank_qM_{\sigma}=n$.
Consequently $\stufe_n|\stufe'_n$.  Now suppose that $q$ is a prime so that
$q|\stufe'_n$ but $q\nmid\stufe_n$.  Thus $n=\rank_qM_{\beta}>\rank_qM_{\sigma}$, as thus as discussed in the proof of Corollary 4.4, we have $\chi_q^2=1$.
Hence $\chi_{\stufe_n}^2=\chi_{\stufe'_n}^2$.
Therefore $\lambda'_{\beta;j}(p^2)=\lambda'_{\sigma;j}(p^2)$ for all $\beta\ge\sigma\ (\sigma).$  Hence $\widetilde\E_{\sigma}|T'_j(p^2)=\lambda'_{\sigma;j}(p^2)\widetilde\E_{\sigma}$.
\end{proof}

\bigskip

\section{Relations on Gauss sums}

To prove the following identities, we frequently use that with $n\times n$ matrices $A,B$, we have $Tr(AB)=Tr(BA).$

\begin{prop}
Suppose $(M\ N),\,(X_sMX_s^{-1}\ X_sNX_s)$ are coprime symmetric pairs.
Then
$$\G_{X_sMX_s^{-1}}(X_sNX_s)=q^s\cdot \G_M(N).$$
\end{prop}

\begin{proof}
We let $U_0$ vary over $\Z^{1,n}/\Z^{1,n}NX_s$, $U_1$ over $\Z^{1,n}/\Z^{1,n}X_s$.
Then $U_1NX_s$ varies over $\Z^{1,n}NX_s/\Z^{1,n}X_sNX_s$; hence
$U=U_0+U_1NX_s$ varies over $\Z^{1,n}/\Z^{1,n}X_sNX_s$ and, recalling that $M\,^tN=N\,^tM$, we have
\begin{align*}
\e\{2\,^tUUX_s^{-1}N^{-1}MX_s^{-1}\}
&=\e\{2\,^tU_0U_0X_s^{-1}N^{-1}MX_s^{-1}\}
\cdot \e\{4\,^tU_0U_1MX_s^{-1}\}.
\end{align*}
%
% Note: $\e\{2\,^tU_1U_0X_s^{-1}N^{-1}M\,^tN\}
% =\e\{2\,^tU_1U_0X_s^{-1}N^{-1}N\,^tM\}=\e\{2MX_s^{-1}\,^tU_0U_1\}$
%
For fixed $U_0$,
$$\sum_{U_1}\e\{4\,^tU_0U_1MX_s^{-1}\} = \sum_{U_1}\{4U_1MX_s^{-1}\,^tU_0\}$$
is a character sum, so the sum is 0 unless $MX_s^{-1}\,^tU_0$ is integral.
Since $X_sMX_s^{-1}$ is integral, $M=\begin{pmatrix} A_1&A_2\\qA_3&A_4\end{pmatrix}$ with $A_1$ of size $s\times s$.
Since 
$$(X_sMX_s^{-1},X_sNX_s)=1,$$ 
we must have that $A_1$ is invertible modulo $q$.
$MX_s^{-1}\,^tU_0$ is integral if and only if $U_0\in\Z^{1,n}X_s$,
from which the proposition follows.
\end{proof}

\begin{prop}  Suppose $(M\ N)$ and $(X_{0,r}MX_r^{-1}\ \,X_{0,r}NX_r)$ are integral, coprime symmetric pairs.  Then
$$\G_{X_{0,r}MX_r^{-1}}(X_{0,r}NX_r)=\G_M(N).$$
\end{prop}

\begin{proof}  First, note that since $(X_{0,r}MX_r^{-1}\ \,X_{0,r}NX_r)$ is an integral coprime pair, we have
$M=\begin{pmatrix}qA_1&A_2\\q^2A_3&qA_4\end{pmatrix}$, $N=\begin{pmatrix} B_1&B_2\\B_3&qB_4\end{pmatrix}$
with $A_3, B_3$ of size $r\times r$ and $B_3$ invertible modulo $q$.

We obtain the desired identity by evaluating in two ways the sum
$$\sum_{U\in\Q^{1,n}/\Z^{1,n}NX_r}\e\{2\,^tUUX_r^{-1}N^{-1}MX_r^{-1}\}.$$

We now show that as $U_0$ varies over $\Z^{1,n}/\Z^{1,n}N$ and $U_1$ varies over $\Z^{1,n}/\Z^{1,n}X_{0,r}^{-1}$,
$U_0X_r+U_1X_{0,r}NX_r$ varies over $\Z^{1,n}/Z^{1,n}NX_r$.  
Define the additive homomorphism $\psi:\Z^{1,n}\times\Z^{1,n}\to \Z^{1,n}/\Z^{1,n}NX_r$ by
$$\psi((U_0,U_1))=U_0X_r+U_1X_{0,r}NX_r+\Z^{1,n}NX_r.$$
Suppose $(U_0,U_1)\in\ker\psi$.  Thus
$$U_0+U_1X_{0,r}N\in\Z^{1,n}N\subseteq\Z^{1,n}.$$
Writing $U_1=(W_1\ W_1')$ where $W_1'$ is $1\times r$, we must have $W_1'\equiv0\ (q)$ since $B_3$ is invertible modulo $q$ and $U_1X_{0,r}N$ is integral.  Hence $U_1X_{0,r}$ is integral, and thus $U_0\in\Z^{1,n}N$.  Thus
$$\ker\psi=\Z^{1,n}N\times\Z^{1,n}X_{0,r}^{-1}.$$
Since
$$|\psi(\Z^{1,n}\times\Z^{1,n})/\ker\psi|=q^r\det N
=|\Z^{1,n}/\Z^{1,n}NX_r|,$$
$\psi$ is an isomorphism.

Thus with $U_0$ varying over $\Z^{1,n}/\Z^{1,n}N$ and $U_1$ over $\Z^{1,n}/\Z^{1,n}X_{0,r}^{-1}$, we have 
\begin{align*}
&\sum_{U\in\Z^{1,n}/\Z^{1,n}NX_r} \e\{2\,^2UUX_r^{-1}N^{-1}MX_r^{-1}\}\\
&=\sum_{U_0,\,U_1} \e\{2\,^t(U_0X_r+U_1X_{0,r}NX_r)(U_0X_r+U_1X_{0,r}NX_r)X_r^{-1}N^{-1}MX_r^{-1}\}\\
&=\sum_{U_0,\,U_1}\e\{2\,^tU_0U_0N^{-1}M\}\,\e\{4\,^tU_0U_1X_{0,r}M\}\,\e\{2\,^tNX_{0,r}\,^tU_1U_1X_{0,r}M\}\\
&=q^r \G_M(N)
\end{align*}
since $X_{0,r}M$, $MX_{0,r}$, and $X_{0,r}M\,^tNX_{0,r}$ are integral.

On the other hand, as $V_0$ varies over $\Z^{1,n}/\Z^{1,n}X_{0,r}NX_r$ and $V_1$ varies over $\Z^{1,n}/\Z^{1,n}X_{0,r}^{-1}$
and hence $V_1X_{0,r}NX_r$ varies over $\Z^{1,n}X_{0,r}NX_r/\Z^{1,n}NX_r$.  So
$V_0+V_1X_{0,r}NX_r$ varies over $\Z^{1,n}/\Z^{1,n}NX_r$.
Thus
\begin{align*}
&\sum_{U\in\Z^{1,n}/\Z^{1,n}NX_r} \e\{2\,^tUUX_r^{-1}N^{-1}M_r^{-1}\}\\
&=\sum_{V_0,V_1}\e\{2\,^tV_0V_0X_r^{-1}N^{-1}MX_r^{-1}\}
\e\{4\,^tV_0V_1X_{0,r}MX_r^{-1}\}\\
& \e\{2\,^tV_1V_1\cdot X_{0,r}M\,^tNX_{0,r}\}\\
&=\sum_{V_0,V_1}\e\{2\,^tV_0V_0X_r^{-1}N^{-1}MX_r^{-1}\}
\end{align*}
since $X_{0,r}M\,^tNX_{0,r}$ is integral.  Thus
$$\sum_{U\in\Z^{1,n}/\Z^{1,n}NX_r} \e\{2\,^2UUX_r^{-1}N^{-1}M_r^{-1}\}
=q^r\G_{X_{0,r}MX_r^{-1}}(X_{0,r}NX_r).$$
This proves the proposition.
\end{proof}

\begin{prop}  Suppose that $(M\ N)$ is a coprime symmetric pair so that
$$M=\begin{pmatrix}qB_1&B_2\\qB_3&qB_4\end{pmatrix},\ 
N=\begin{pmatrix}C_1&C_2\\C_3&qC_4\end{pmatrix}$$
where $B_3, C_3$ are $\ell\times\ell$ and invertible modulo $q$.
Then $\rank_q(B_2\ C_2)=n-\ell$, $(MX_{\ell}^{-1}\ \,NX_{\ell})$ is a coprime
symmetric pair, and
$$\G_{MX_{\ell}^{-1}}(NX_{\ell})=
\left(\frac{\det B_3C_3}{q}\right) (\G_1(q))^{\ell} \G_M(N).$$
\end{prop}

\begin{proof}
Since $C_3$ is invertible modulo $q$, we have
$$n=\rank_q\begin{pmatrix}B_2&0&C_2\\0&C_3&0\end{pmatrix}
=\rank_q\begin{pmatrix}B_2&0&C_2\\0&C_3&0\end{pmatrix},$$
hence $\rank_q(B_2\ C_2)=n-\ell$.
Also, $\rank_q(MX_{\ell}^{-1}\ \,NX_{\ell})=\rank_q
\begin{pmatrix}0&B_2&C_2\\B_3&0&0\end{pmatrix}=n,$
so $(MX_{\ell}^{-1}\ \,NX_{\ell})$ is a coprime symmetric pair.
We know that $X_{0,\ell}NX_{\ell}$ is integral, so we define the additive homomorphism 
$\psi:\Z^{1,n}\times\Z^{1,n}\to \Z^{1,n}/\Z^{1,n}NX_{\ell}$ by
$$\psi((U_0,U_1))=U_0X_{\ell}+U_1X_{0,\ell}NX_{\ell}+\Z^{1,n}NX_{\ell}.$$
Then just as proved in Proposition 5.2, $\psi$ is surjective with kernel
$\Z^{1,n}N\times\Z^{1,n}X_{0,\ell}NX_{\ell}.$
Thus, as $X_{0,\ell}M$ is integral, we find that
\begin{align*}
&\G_{MX_{\ell}^{-1}}(NX_{\ell})
% \\&\qquad =
%\sum_{U_0\in\Z^{1,n}/\Z^{1,n}N}\e\{2\,^tU_0U_0N^{-1}M\}
%\sum_{U_1\in\Z^{1,n}/\Z^{1,n}X_{0,\ell}^{-1}}\e\{2\,^tU_1U_1X_{0,\ell}M\,^tNX_{0,%\ell}\}\\
%&\qquad
=\G_M(N)\G_{M\,^tNX_{0,\ell}}(X_{0,\ell}^{-1}).
\end{align*}

To evaluate $\G_{M\,^tNX_{0,\ell}}(X_{0,\ell}^{-1})$, we first note that
$\{(0\ V):\ V\in\Z^{1,\ell}/q\Z^{1,\ell}\ \}$
is a set of representatives for $\Z^{1,n}/\Z^{1,n}X_{0,\ell}^{-1}$.
Thus
$$\G_{M\,^tNX_{0,\ell}}(X_{0,\ell}^{-1})
=\sum_{V\in\Z^{1,\ell}/q\Z^{1,\ell}}\e\{2\,^tVVB_3\,^tC_3/q\}.$$
Since $q\not=2$ and $B_3\,^tC_3$ is symmetric and invertible modulo $q$, by section 2.8 \cite{Ger}, there is some $G\in SL_{\ell}(\Z)$ and $w_1,\ldots,w_{\ell}\in\Z$
 so that 
$$GB_3\,^tC_3\,^tG\simeq\big<w_1,w_2,\ldots,w_{\ell}\big>$$
 where $q\nmid w_1w_2\cdots w_{\ell}$.
Since $\Z^{1,\ell}G/q\Z^{1,\ell}G=\Z^{1,\ell}/q\Z^{1,\ell}$, replacing $V$ by $VG$ gives us
\begin{align*}
\G_{M\,^tNX_{0,\ell}}(X_{0,\ell}^{-1})
&= \sum_{v_1,\ldots,v_{\ell}\,(q)}
\e\{2(v_1^2w_1+v_2^2w_2+\cdots+v_{\ell}^2w_{\ell})/q\}\\
&= \left(\frac{w_1w_2\cdots w_{\ell}}{q}\right) \left(\G_1(q)\right)^{\ell}\\
&= \left(\frac{\det B_3C_3}{q}\right)\left(\G_1(q)\right)^{\ell},
\end{align*}
proving the proposition.
\end{proof}


\begin{thebibliography}{0}



\bibitem{And} A.N. Andrianov, \emph{Quadratic Forms and Hecke Operators}, Springer-Verlag,  1987.



% \bibitem{B} S. B\"ocherer, ``\"Uber die Fourierkoeffizienten der Siegelschen Eisensteinreihen"
% \textit{ Manuscripta Math. } 45 (1984), 273-288.

% \bibitem{B2} S. B\"ocherer, ``On the Hecke operator U(p)", with an appendix by Ralf %Schmidt.
% \textit{ J. Math. Kyoto Univ.} 45 (2005), no. 4, 807–829.

%\bibitem{B3} S. B\"ocherer, ``On the space of Eisenstein series for $\Gamma_0(p)$: Fourier %expansions", with an appendix by H. Katsurada.
%\textit{ Comment. Math. Univ. St Pauli} 63 (2014), no. 1-2, 3-22.


% \bibitem {CK} Y. Choie, W. Kohnen, ``Fourier coefficients of Siegel-Eisenstein series of odd genus"
% \textit{ J. Math. Anal. Appl. } 374 (2011), 1-7.


%\bibitem{Dic} M. Dickson, ``Fourier coefficients of degree two Siegel-Eisenstein series with trivial %character at squarefree level ".
%\textit{Ramanajun J.} 37 (2015), no. 3, 541-562.

%\bibitem{E1} S.A. Evdokimov, ``A basis composed of eigenfunctions of Hecke operators in the %theory of modular forms of genus $n$". (Russian) \textit{ mat. Sb. (N.S.)} 115(157) (1981), %no. 3, 337-363.

%\bibitem{E2} S.A. Evdokimov, Letter to the editors: ``A basis composed of eigenfunctions of %Hecke operators in the theory of modular forms of genus $n$". (Russian) \textit{ mat. Sb. %(N.S.)} 116(158) (1981), no. 4, 603.


%\bibitem {F} E. Freitag, ``Siegel Eisenstein series of arbitrary level and theta series".
5\textit{ Abh. Math. Sem. Univ. Hamburg } 66 (1996), 229-247.


\bibitem {Ger} L. Gerstein, ``Basic Quadratic Forms". \textit { Graduate
Studies in Math.} Vol. 90, Amer. Math. Soc., 2008.

\bibitem {HW} J.L. Hafner, L.H. Walling, ``Explicit action of Hecke operators
on Siegel modular forms". \textit{ J. Number Theory } 93 (2002), 34-57.


% \bibitem {Kat1} H. Katsurada, ``An explicit formula for the Fourier coefficients of 
%Siegel-Eisenstein series
% of degree 3" \textit{ Nagoya Math. J. } 146 (1997), 199-223.


% \bibitem {Kat2} H. Katsurada, ``An explicit formula for Siegel series"
% \textit{ Am. J. Math. } 121(2) (1999), 415-452.

%\bibitem {Klo} K. Klosin, ``A note on Hecke eigenvalues of hermitian Siegel Eisenstein series".
%\textit{ Ramanujan J. } 35 (2014), no. 2, 287-298.


% \bibitem {Koh} W. Kohnen, ``Lifting modular forms of half-integral weight to Siegel modular forms
% of even genus" \textit{ Math. Ann. } 322 (2002), 787-809.


% \bibitem {Maass1} H. Maass, ``Die Fourierkoeffizienten der Eisensteinreihen zweiten Grades"
% \textit{ Mat.-Fys. Medd. Danske Vid. Selsk. } 34 (7),  (1964), 25 p.


% \bibitem {Maass2} H. Maass, ``\"Uber die Fourierkoeffizienten der Eisensteinreihen zweiten Grades"
% \textit{ Mat.-Fys. Medd. Danske Vid. Selsk. } 38 (14), (1972), 13 p.


% \bibitem {Miz} Y. Mizuno, ``An explicit arithmetic formula for the Fourier coefficients of Siegel-Eisenstein
% series of degree two and square-free odd levels" \textit{ Math. Z. } 263 (2009), 837-860.

%\bibitem {Ogg} A. Ogg, \emph{Modular forms and Dirichlet Series}, Benjaman Press, New %York, 1969.

\bibitem{O'M} O.T. O'Meara, \emph{Introduction to Quadratic Forms}, Springer-Verlag, 1987.

\bibitem {Shim} G. Shimura, ``On modular forms of half integral weight." \textit{Annals of Math.} 97 (1973), 440-481.

% \bibitem{Tak} S. Takemori, ``$p$-adic Siegel-Eisenstein series of degree 2"
% \textit{ International J. Number Theory } ???.

% \bibitem{thetaI} L.H. Walling, ``Action of Hecke operators on Siegel theta series I".
% \textit{ International J. of Number Theory } 2 (2006), 169-186.

% \bibitem{thetaII} L.H. Walling, ``Action of Hecke operators on Siegel theta series II".
% \textit{ International J. of Number Theory } 4 (2008), 981-1008.


% \bibitem{Wal} L.H. Walling, ``Hecke eigenvalues and relations for degree 2 Siegel Eisenstein 
% series".
% \textit{ J. Number Theory } 132 (2013), 2700-2723.


\bibitem{half-int} L.H. Walling, ``A formula for the action of Hecke operators on half-integral weight Siegel modular forms and applications."  \textit{J. Number Theory} 133 (2013), 1608-1644.

\bibitem{int wt} L.H. Walling, ``Hecke eigenvalues and relations for Siegel Eisenstein series of arbitrary degree, level, and character."  \textit{International J. Number Theory} (to appear).

\end{thebibliography}
\end{document}